\documentclass{article}
\usepackage{graphicx} 
\usepackage{amsmath, amssymb, amsthm}
\usepackage{graphicx}
\usepackage{bm}
\usepackage{upgreek}
\usepackage{enumitem}
\usepackage{hyperref}
\hypersetup{colorlinks,linkcolor={blue},citecolor={red},urlcolor={blue}}
\usepackage{cleveref}
\usepackage{mathtools}

\newtheorem{assumptionA}{\textbf{A}\hspace{-3pt}}
\Crefname{assumptionA}{\textbf{A}\hspace{-3pt}}{\textbf{H}\hspace{-3pt}}
\crefname{assumptionA}{\textbf{A}}{\textbf{A}}

\newtheorem{assumptionB}{\textbf{B}\hspace{-3pt}}
\Crefname{assumptionB}{\textbf{B}\hspace{-3pt}}{\textbf{H}\hspace{-3pt}}
\crefname{assumptionB}{\textbf{B}}{\textbf{B}}

\usepackage{graphicx}
\graphicspath{ {./images/} }
\usepackage{float}

\newtheorem{thm}{Theorem}

\newtheorem{lemma}{Lemma}

\newtheorem{prop}{Proposition}

\newtheorem{note}{Note}
\newtheorem{rem}{Remark}
\usepackage[letterpaper,top=2cm,bottom=2cm,left=3cm,right=3cm,marginparwidth=1.75cm]{geometry}

\usepackage{authblk}

\title{Convergence Of The Unadjusted Langevin Algorithm For Discontinuous Gradients}
\author[1]{Tim Johnston}
\author[1,2,3]{Sotirios Sabanis}
\affil[1]{\textit{University of Edinburgh, United Kingdom.}}
\affil[2]{\textit{The Alan Turing Institute, United Kingdom.}}
\affil[3]{\textit{National Technical University of Athens, Greece.}}
\date{\today}
\usepackage{graphicx}
\graphicspath{ {./Diagrams/} }
\usepackage{float}

\begin{document}

\maketitle
\begin{abstract}
   We demonstrate that for strongly log-convex densities whose potentials are discontinuous on manifolds, the ULA algorithm converges with stepsize bias of order $1/2$ in Wasserstein-p distance. Our resulting bound is then of the same order as the convergence of ULA for gradient Lipschitz potential. Additionally, we show that so long as the gradient of the potential obeys a growth bound (therefore imposing no regularity condition), the algorithm has stepsize bias  of order $1/4$. We therefore unite two active areas of research: i) the study of numerical methods for SDEs with discontinuous coefficients and ii) the study of the non-asymptotic bias of the ULA algorithm (and variants). In particular this is the first result of the former kind we are aware of on an unbounded time interval.
\end{abstract}

\section{Introduction}
In this paper we consider the overdamped Langevin SDE
\begin{equation}\label{eq: main SDE}
dY_t = -\nabla U(Y_t) dt+\sqrt{\frac{2}{\beta}}dW_t, \;\;\; t\geq 0,
\end{equation}
where $U:\mathbb{R}^d\to\mathbb{R}$ is a potential, $\beta>0$ is the `inverse temperature parameter' and $(W_t)_{t\geq 0}$ is an $\mathbb{R}^d$-valued Wiener martingale independent of the initial condition $\xi$. It is well known that under weak conditions \eqref{eq: main SDE} admits the unique invariant measure $\pi_\beta \sim e^{-\beta U}$. Given this fact, the overdamped Langevin diffusion has been used as a basis for a variety of sampling and optimization algorithms. In this paper we consider the Unadjusted Langevin Algorithm, or ULA, which is given by the Euler-Maruyama discretisation of \eqref{eq: main SDE}, or specifically
\begin{equation}\label{eq: ULA}
x_{n+1}=x_n-\gamma \nabla U(x_n)+\sqrt{\frac{2\gamma}{\beta}}z_{n+1}, \;\;\; x_0=\xi, \;\;\; n\in \mathbb{N},
\end{equation}
where $\gamma>0$ is the stepize and $(z_n)_{n\geq 1}$ is a sequence of iid standard Gaussians on $\mathbb{R}^d$. This algorithm was first proposed in a physical context in \cite{PARISI1981378} and \cite{10.1063/1.430300}, and in the context of image recognition in \cite{400cc48a-c98e-3a92-ac96-5477dc6f3a71}. It has been shown to be effective for sampling from Bayesian posteriors, see \cite{NIPS1992_f29c21d4, 0b5e028f-1a33-3fed-ae08-229fd5443c3f, articleula}, and utilised as part of a marginal maximal likelihood algorithm in \cite{DeBortoli2021}. Theoretical properties of the algorithm have been studied under a Lipschitz assumption on $\nabla U$ and suitable ergodicity properties in \cite{0b5e028f-1a33-3fed-ae08-229fd5443c3f, strathprints171}, and especially thoroughly under the assumption of strong convexity in \cite{articleula, Dalalyan2017FurtherAS}. A popular variant of the ULA algorithm where the gradient is replaced by an estimate is known as Stochastic Gradient Langevin Dynamics, see \cite{Welling2011BayesianLV, DALALYAN20195278}, and has been analysed under very general conditions in \cite{pmlr-v65-raginsky17a, 3b99fee3596546ae809bc62b56a361be, Zhang2023}. Furthermore these methods have been extended to constrained problems via the so called `projected stochastic gradient Langevin algorithm', given in \cite{osti_10354811}.

However, even though the ULA algorithm has been adapted to when the gradient is continuous but not globally Lipschitz in \cite{BROSSE20193638}, we are not aware of any literature on the convergence properties of \eqref{eq: ULA} when $\nabla U$ is not even continuous. The closest thing we are aware of is the strategy of \cite{doi:10.1137/16M1108340, Luu2021, Pereyra2016}, where a discontinuous gradient is smoothed via the computationally-intensive Moreau-Yosida regularisation. Additionally, the case of SGLD with discontinuous stochastic gradient was considered in \cite{lim2023langevin, f73ade6c0bbe4134b7a98b3db5029e80, RePEc:arx:papers:2007.01672}, however with the assumption of a `continuity in average' condition which excludes truly discontinuous densities.

On the other hand, in the context of numerical methods for SDEs (on bounded time intervals), there has been a flurry of research in recent years addressing this issue, see \cite{szolgyenyi2021stochastic} for a survey. Two main lines of research have appeared: Strategy I, for which $L^p$ convergence rate of $1/2$ has been proven for SDEs with piecewise-Lipschitz drift and multiplicative noise in \cite{articletl}, and Strategy II, for which an $L^2$ convergence rate of $1/2-$ has been proven for SDEs with additive noise and drift that is merely bounded and measurable in \cite{articleboundedmeas}. The technique of Strategy I involves a judicious choice of `transformation function' which smooths the discontinuity of the drift, in addition to bounds on the intervals in which the scheme crosses a point of discontinuity. It has been expanded to the case where the drift is superlinear in \cite{müllergronbach2022existence}, to the multidimensional case (where the discontinuities lie on manifolds) in \cite{multidimescheme}, and to a convergence rate $3/4$ when the drift is piecewise smooth in \cite{10.1093/imanum/draa078}. The technique of Strategy II involve involve the regularising properties of the noise, and been extended to the case of Levy processes in \cite{butkovsky2022strong}, multiplicative noise in \cite{10.1214/22-AAP1867} and SPDEs in \cite{doi:10.1137/21M1454213}. These generalisations in particular make use the stochastic sewing lemma established in \cite{10.1214/20-EJP442}. We note that whilst Strategy II is more general, it has not yet been adapted to the case where the coefficients of the SDE are unbounded. Lower bounds which demonstrate the optimality of many of results from Strategy I and II results have been established in \cite{nonliplowerbound, 10.1214/22-AAP1837}.


Therefore, in the present work we unite these two lines of research and prove in Theorem \ref{th: main theorem} that the ULA algorithm \eqref{eq: ULA} converges with $L^p$ discretisation error of order $1/2$ (uniformly in time). We assume that $U$ is strongly convex, and Lipschitz outside of a collection of sufficiently smooth compact hypersurfaces. Whilst we use techniques inspired in part by the strategy of \cite{articletl}, due to monotonicity of $\nabla U$ we shall not need to use a transformation function, which means the collection of hypersurfaces considered can intersect in an arbitrary manner. The main technical challenge then is to extend the work discussed in the previous paragraph to an unbounded time interval. To do this we shall multiply the difference process between the scheme and the true solution by an appropriate exponential and apply Proposition \ref{prop: exponential crossing}, which demonstrates that the scheme does not cross the hypersurfaces of discontinuity too often (weighted by an exponential function of time). A key ingredient in the proof of Proposition \ref{prop: exponential crossing} is the occupation time formula Lemma \ref{lem: time dep local time}, which is a corollary of the classical local time identity.

One disadvantage of Theorem \ref{th: main theorem} is that it requires \Cref{assmp: piecewise Lip}, which states that the discontinuities lie on a compact manifold. This therefore rules out many relevant examples where the discontinuity lies on an unbounded manifold, say a hyperplane. This is the case for instance for the examples considered in \cite{doi:10.1137/16M1108340}. However, using a much simpler argument than in Theorem \ref{th: main theorem}, we show in Theorem \ref{th: main theorem 2} that no matter how irregular $\nabla U$ is, so long as it obeys a linear growth bound the numerical error of the Euler scheme in $L^p$ is of order $1/4$ uniformly in time. This therefore demonstrates that the Euler scheme is broadly `robust' in the case where $\nabla U$ is convex, in the sense that the discretisation error does indeed converge to $0$ at polynomial rate as the stepsize tends to $0$. This is in contrast to the general case for the Euler scheme, where even for bounded coefficients the scheme can converge to the true solution arbitrarily slowly, see \cite{Hairer2012LossOR, doi:10.1080/07362994.2016.1263157}.

\subsection{Assumptions}\label{subsec: assmp}

We assume strong convexity of $U$ (which implies strong monotonicity of $\nabla U$) in \Cref{assmp: strong mon} below, as well as the integrability of the initial condition $\xi$. For the regularity of $\nabla U$ we have two different assumptions. For Theorem \ref{th: main theorem} we have \Cref{assmp: piecewise Lip}, in which we assume there are open sets $\Phi_j \subset \mathbb{R}^d$ (the union of whose completion is $\mathbb{R}^d$) on which $\nabla U$ is (piecewise) Lipschitz continuous, and that the boundary of these open sets is a subset of a collection of sufficiently regular compact hypersurfaces. We do not require that $\nabla U$ be well defined on the boundary of the $\Phi_j$, and indeed $\nabla U$ will never be well defined at a point of discontinuity. In Proposition \ref{prop: well posedness} we prove that all processes of interest are well defined none the less. Note that unlike in \cite{multidimescheme}, one does not require that the union of the hypersurfaces $\Sigma_i$ be smooth overall.

As an alternative to \Cref{assmp: piecewise Lip} we have the much weaker assumption \Cref{assmp: growth assumption} for Theorem \ref{th: main theorem 2}, which places no requirement on the regularity of $\nabla U$ at all besides that it is well-defined almost everywhere and does not grow superlinearly. This weaker assumption comes at the cost of a weaker numerical error of order $1/4$.

Note that in \Cref{assmp: initial condition} we assume $\xi$ is such that $\nabla U(\xi)$ is almost surely well defined (so that in particular $\xi$ takes values in $\cup^{n_2}_{j=1}\Sigma_j$ with probability $0$). This is necessary since otherwise the first iteration of \eqref{eq: ULA} will not be well defined. This however causes us no problems beyond the initial condition, as we show in Proposition \ref{prop: well posedness}.

\begin{assumptionA}\label{assmp: piecewise Lip}
\textbf{(Piecewise Continuity)} There exist bounded regions $M_j \subset \mathbb{R}^d$ and compact, orientable, connected, $C^3$ hypersurfaces $\Sigma_j\subset \mathbb{R}^d$ such that $\Sigma_j=\partial M_j$ for $j=1,2,...,n_2$. Furthermore these hypersurfaces cut $\mathbb{R}^d$ into smaller regions, that is, there exist disjoint open sets $\Phi_i \subset \mathbb{R}^d$, $i=1,2,...,n_1$ such that $\cup_{i=1}^{n_1} \overline{\Phi}_i=\mathbb{R}^d$ and $\cup_{i=1}^m\partial \Phi_i \subset \cup_{j=1}^{n_2} \Sigma_j$. Then $\nabla U$ is piecewise-Lipschitz on the $\Phi_i$. Specifically, there exists $L>0$ such that
\begin{equation}
\lvert \nabla U(x)-\nabla U(y)\rvert \leq L\lvert x-y\rvert, \;\;\;x,y \in \Phi_i,\;\;\; i=1,2,...,n_1.
\end{equation}
\end{assumptionA}

\begin{assumptionB}
\label{assmp: growth assumption}
\textbf{(Growth Assumption)} The function $\nabla U: \mathbb{R}^d\to\mathbb{R}^d$ exists almost everywhere, and there exists $L, m>0$ such that
\begin{equation}
\lvert \nabla U (x)\rvert \leq m+L\lvert x \rvert.
\end{equation}
\end{assumptionB}

\begin{assumptionA}\label{assmp: strong mon}
\textbf{(Strong Monotonicity)} There exists $\mu>0$ such that
\begin{equation}\label{eq: strong mon condition}
\langle \nabla U(x)-\nabla U(y), x-y\rangle \geq \mu\lvert x-y\rvert^2, \;\;\; x,y \in \mathbb{R}^d.
\end{equation}
\end{assumptionA}

\begin{assumptionA}\label{assmp: initial condition}
\textbf{(Integrable Initial Condition)} One has $P(\xi \in \{x \in \mathbb{R}^d \vert \; \nabla U(x) \; \text{well defined}\})=1$. Additionally, for every $p>0$ one has
\begin{equation}
E\lvert \xi \rvert^p<\infty
\end{equation}
\end{assumptionA}

Since each of the $\Sigma_j$ are compact, there exists an $R>0$ large enough that $\cup ^{n_2}_{i=1} \Sigma^\delta_i \subset B_R$. Therefore only one of the $\Phi_i$ is unbounded, and we can assume without loss of generality that it is $\Phi_1$ is unbounded.

\begin{figure}[H]\centering
\includegraphics[width=6.3cm, height=4.4cm]{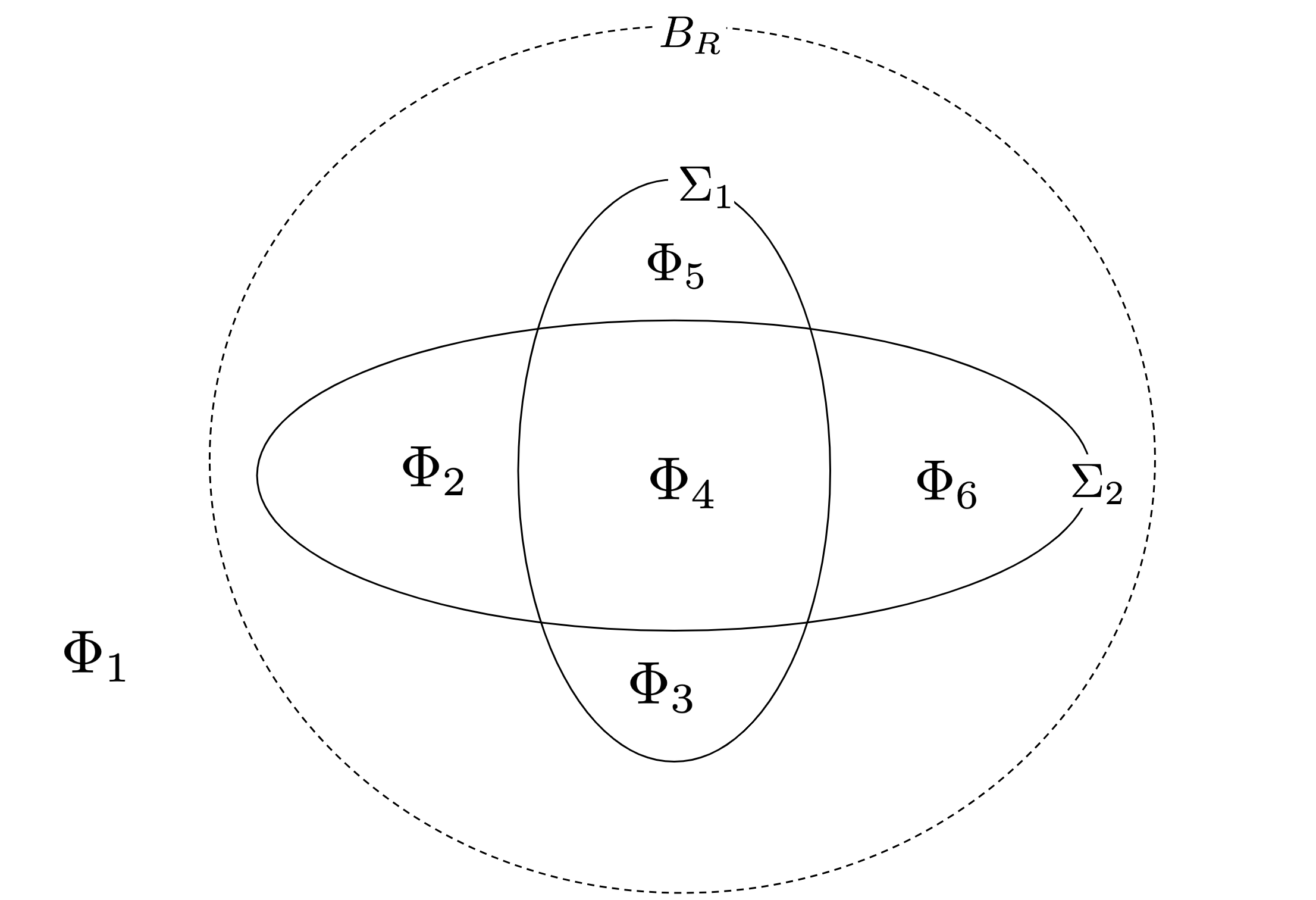}
\caption{Under \Cref{assmp: piecewise Lip} we have that the hypersurfaces $\Sigma_j$ cut $\mathbb{R}^d$ into regions $\Phi_i$ on which $\nabla U$ is continuous. Additionally, $R>0$ is large enough that $\cup ^{n_2}_{i=1} \Sigma^\delta_i \subset B_R$, and of the $\Phi_i$ only $\Phi_1$ is unbounded. Note that in this diagram $M_j$, $j=1,2$ are the interiors of the two hypersurfaces, so that $M_1 = \Phi_3 \cup \Phi_4 \cup \Phi_5$ and $M_2=\Phi_2\cup \Phi_4\cup\Phi_6$.}\label{visina7}
\end{figure}

\subsection{Well-Posedness and Set-Up}\label{subsec: setup}
In this section we show that the problem is well-posed despite the fact $\nabla U$ is only defined almost everywhere (and in particular, is not in general defined on the $\Sigma_i$). Firstly let us first define the continuous interpolation of \eqref{eq: ULA}. Let $\kappa_\gamma(t):=\gamma\lfloor \frac{t}{\gamma} \rfloor$ be the backwards projection onto the grid $\{0,\gamma, 2\gamma,...\}$, so that $\kappa_\gamma(n\gamma+\epsilon)=n\gamma$ for $n \in \mathbb{N}$ and $\epsilon \in [0,\gamma)$. Furthermore let $\underline{\kappa}_\gamma(t):=\kappa_\gamma(t)+\gamma$ be the forward projection onto $\{0,\gamma, 2\gamma,...\}$. Then one may define
\begin{equation}\label{eq: cont interp}
X_t = \xi-\int^t_0 \nabla U(X_{\kappa_\gamma(s)}) ds+\sqrt{\frac{2}{\beta}}W_t.
\end{equation}
Setting $z_n:=\frac{W_{n\gamma}-W_{(n-1)\gamma}}{\sqrt{\gamma}}$ one then sees by this definition that $X_{n\gamma}=x_n$ for all $n\in \mathbb{N}$. From now on it will mostly be more convenient to analyse \eqref{eq: cont interp}.

\begin{prop}\label{prop: well posedness}
    Let \Cref{assmp: growth assumption} and \Cref{assmp: initial condition} hold. Then for every non-random initial condition $\xi=x\in \mathbb{R}^d\setminus \cup^{n_2}_{j=1}\Sigma_j$ the SDEs \eqref{eq: main SDE} and \eqref{eq: cont interp} have well-defined unique strong solutions.
\end{prop}
\begin{proof}
    The unique strong solution of \eqref{eq: main SDE} follows from Theorem 2.1 in \cite{ZHANG20051805}. For \eqref{eq: cont interp} to be well defined, by the structure of the Euler scheme we just have to show that $P(X_{n\gamma} \in \{x \in \mathbb{R}^d \vert \; \nabla U(x) \; \text{well defined}\})=1$ for every $n\in \mathbb{N}$ and initial condition $x\in  \mathbb{R}^d\setminus \cup^{n_2}_{j=1}\Sigma_j$. Note that $\mathcal{L}(X_{n\gamma})=\mathcal{L}(x_n)$, so it sufficient to check $P(x_n\in \cup_{i=1}^{n_1}\Sigma_i)=0$ for every $n\geq 0$. Since $x_0=\xi$ this holds by \Cref{assmp: initial condition} for $n=0$. Then it holds for $n\geq 1$ since the law of $x_n$ conditioned on $x_{n-1} \in \mathbb{R}^d\setminus \cup^{n_2}_{j=1}\Sigma_j$ is absolutely continuous with respect to the Lesbesgue measure for every $n\geq 1$.
\end{proof}
Since \eqref{eq: main SDE} and \eqref{eq: cont interp} have strong solutions for every non-random initial condition $\xi=x\in \mathbb{R}^d \setminus \cup^{n_2}_{j=1}\Sigma_j$, one may set up the probability space for $(Y_t)_{t\geq 0}$ and $(X_t)_{t\geq 0}$ as follows. Let $(\Omega_1, \mathcal{H}^1, P_1 )$ be a probability space on which $\xi$ is defined and let $(\Omega_2, (\mathcal{H}^2_t)_{t\geq 0}, P_2)$ be a filtered probability space on which $(W_t)_{t\geq 0}$ is defined and adapted to. Then one can define the probability space $(\Omega,(\mathcal{F}_t)_{t\geq 0}, P )$ as the product of the two probability spaces, so that $\Omega=\Omega_1\times \Omega_2$, $dP(\omega)=dP(\omega_1,\omega_2)=dP_1(\omega_1)dP_2(\omega_2)$ and $\mathcal{F}_t= \mathcal{H}^2_t\times \mathcal{H}^1$. Then for $x \in \mathbb{R}^d\setminus \cup^{n_2}_{j=1}\Sigma_j$, the map $\omega=(\omega_1, \omega_2)\mapsto X^x_t(\omega_2)$ can be defined as the strong solution to
\begin{equation}\label{eq: cont interp x}
X^x_t = x-\int^t_0 \nabla U(X^x_{\kappa_\gamma(s)}) ds+\sqrt{\frac{2}{\beta}}W_t,
\end{equation}
and under \Cref{assmp: initial condition}, one has $X_t(\omega)=X_t(\omega_1, \omega_2):=X^{\xi(\omega_1)}_t(\omega_2)$. Then for all measurable functionals $f$ from measurable functions into $\mathbb{R}$, one may define 
\begin{equation}
E^x f((X_t)_{t\geq 0}):=E f((X^x_t)_{t\geq 0}) ,
\end{equation}
and additionally for all Borel-measurable subsets $A$ of functions $\mathbb{R}^d\to\mathbb{R}$, one may define $P^x((X_t)_{t\geq 0} \in A):=P((X^x_t)_{t\geq 0} \in A)$. One can similarly define the process $(Y^x_t)_{t\geq 0}$ and extend these definitions to \eqref{eq: main SDE}. Therefore $E^x$ corresponds to the conditional definition $E^x[\cdot]=E [\cdot  \vert \; \xi=x]$. Furthermore by Fubini's Theorem, given the integrand is uniformly integrable and $P(\xi \in \mathbb{R}^d \setminus \cup^{n_2}_{j=1}\Sigma_j)=0$ one has
\begin{equation}\label{eq: initial conditioin Fubini}
    E[\cdot] = \int_{\mathbb{R}^d\setminus \cup^{n_2}_{j=1}\Sigma_j}E^x[\cdot]d\mu_\xi(x),
\end{equation}
where $\mu_\xi$ is the law of $\xi$. It follows that if, for any random variable $z$, we may define the random operator $E^z[\cdot]$ by 
\begin{equation}\label{eq: E^z definition}
\omega\mapsto E^{z(\omega)}[\cdot]. 
\end{equation}
Then one has
\begin{lemma}\label{lemma: Markov property}
Consider the definition \eqref{eq: E^z definition}. Then for every $n\in \mathbb{N}$, $t\geq 0$ and measurable function $f:\mathbb{R}^d\to\mathbb{R}$ one has
\begin{equation}
    E^{X^x_{n\gamma}}[f(X_t)]=E^x[f(X_{t+n\gamma}) \vert \mathcal{F}_{n\gamma}],\;\;\;E^{X_{n\gamma}}[f(X_t)]=E[f(X_{t+n\gamma}) \vert \mathcal{F}_{n\gamma}].
\end{equation}
\end{lemma}
\begin{proof}
    Due to the discrete Euler Scheme structure of $X_t$, for every $x\in \mathbb{R}^d$ and $t\geq 0$ there exists $m \in \mathbb{N}$ and a function $u^x_t:(\mathbb{R}^d)^m\to\mathbb{R}^d$ such that 
    \begin{equation}
    u^x_t(W_\gamma, W_{2\gamma}-W_{\gamma},...,W_{\kappa_\gamma(t)}-W_{\kappa_\gamma(t)-\gamma},W_t-W_{\kappa_\gamma(t)})=X^x_t,
    \end{equation}
    almost surely (where the last argument is obviously supressed if $t=\kappa_\gamma(t)$). Therefore
    \begin{align}
        E^x[f(X_{t+n\gamma}) \vert \mathcal{F}_{n\gamma}]&= E[f(u^x_{t+n\gamma}(W_\gamma, W_{2\gamma}-W_{\gamma},...,W_{t+n\gamma}-W_{\kappa_\gamma(t)+n\gamma})) \vert \mathcal{F}_{n\gamma}]\nonumber \\
        &=\int_{\Omega} f(u^x_{t+n\gamma}(W_\gamma, W_{2\gamma}-W_{\gamma},...,W_{n\gamma}-W_{(n-1)\gamma},z))d\mu(z),
    \end{align}
    where we have compressed all arguments of $u^x_{t+n\gamma}$ after the $n$th into $z$, and denoted the respective law $\mu$. The fact that we can take expectation of over all random inputs independent of $\mathcal{F}_{n\gamma}$ follows in the same way as Example 4.1.7 from \cite{DU04}. Furthermore observe that by the Markov property of $X_0, X_\gamma, X_{2\gamma}$ one has
    \begin{equation}
       u^x_{t+n\gamma}(w_1,..., w_{n+m})=u_t^{u^x_{n\gamma}(w_1,...,w_n)}(w_{n+1},...,w_{n+m}),
    \end{equation}
    so that
     \begin{equation}
        E^x[f(X_{t+n\gamma}) \vert \mathcal{F}_{n\gamma}](\omega)= \int_{\Omega} f(u^{X^x_{n\gamma}(\omega)}_t(z))d\mu(z)= E^{X^x_{n\gamma}(\omega)}[f(X_t)].
    \end{equation}
    The first result then follows. For the second, one integrates $x$ with respect to the law of $\xi$, the initial condition of $X_t$.
\end{proof}


\subsection{Delta-Neighbourhood}\label{subsec: delta nbhd}
Essential for the following arguments shall be the constant $\delta>0$ given below, which provides a good neighbourhood of the $\Sigma_i$ on which to perform our analysis. Let us first define for $j=1,2,...,n_2$ the signed distance function
\begin{equation}\label{eq: rho i def}
\rho_j(x)= \begin{cases}
dist(x, \Sigma_j),\;\;\;\;\;x \in \mathbb{R}^d\setminus M_j \\
0,\;\;\;\;\;\;\;\; x \in \Sigma_j \\
-dist(x, \Sigma_i),\;\;\; x \in M_j,
\end{cases}
\end{equation}
and, for $\epsilon>0$, the epsilon neighbourhood
\begin{equation}\label{eq: eps neighbourhood def}
\Sigma^{\epsilon}:= \{ x \in \mathbb{R}^d \; \vert \; dist(x, \Sigma) < \epsilon \}.
\end{equation}
\begin{prop}\label{prop: good delta}
There exists a constant $\delta>0$ such that for every $j=1,2,...,n_2$
\begin{enumerate}[label=(\roman*)]
\item one has that $\rho_j\in C^2(\Sigma_j^{3\delta})$
\item for every $x\in \Sigma_j^{3\delta}$, one has $\lvert \nabla \rho_j(x)\rvert =1$
\end{enumerate}
\end{prop}
\begin{proof}
The first part follows from the statement of \cite{gilbarg2001elliptic}, Lemma 14.16, choosing $\delta=\mu/3$ for $\mu$ as in the Lemma. The second part follows from the last two lines of the proof, where it is shown that for every $x\in \Sigma_j^{3\delta}$ there exists $y\in \Sigma_j$ such that $\nabla \rho_j (x) =n_j(y)$, where $n_j$ is the unit normal to $\Sigma_j$. Therefore $\lvert \nabla \rho_j (x) \rvert =\lvert n_j(y)\rvert=1$.
\end{proof}
It follows that if we define $\phi \in C^\infty(\mathbb{R})$ to be a smooth cutoff function satisfying $\phi=1$ on $[-\delta, \delta]$ and $\phi=0$ outside of $(-2\delta, 2\delta)$, one has that $P_j:=\phi \circ\rho_j$ satisfies $P_j=\rho_j$ on $\Sigma_j^\delta$ and $P_j\in C^2(\mathbb{R}^d)$. This will be useful in order to apply Ito's formula in the proof of Lemma \ref{lemma: local time dif}.

\subsection{Main Theorems}
Now we present our main results: note that replacing the piecewise Lipschitz assumption \Cref{assmp: piecewise Lip} with the weaker growth assumption \Cref{assmp: growth assumption} leads to a weakening of the order of numerical error from $1/2$ to $1/4$.

\begin{thm}\label{th: main theorem}
Let \Cref{assmp: piecewise Lip}, \Cref{assmp: strong mon} and \Cref{assmp: initial condition} hold. Let $p\geq 1$ and $\beta>0$. Consider $\pi_\beta \sim e^{-\beta U}$, the invariant measure of the overdamped Langevin diffusion \eqref{eq: main SDE}. Let $\gamma_0 \in (0, \frac{\mu}{2L^2})$. Then there exists $c>0$ such that for every $\gamma \in (0,\gamma_0)$ the unadjusted Langevin algorithm (ULA) given in \eqref{eq: ULA} satisfies
\begin{equation}
W_p(\pi_\beta, \mathcal{L}(x_n)) \leq W_p(\xi, \pi_\beta)e^{-\mu \gamma n}+c\gamma^{1/2}.
\end{equation}
\end{thm}
\begin{thm}\label{th: main theorem 2}
Let \Cref{assmp: growth assumption}, \Cref{assmp: strong mon} and \Cref{assmp: initial condition} hold. Let $p\geq 1$, $\beta>0$, $\pi_\beta $ and $\gamma_0>0$ be as above. Then there exists $c>0$ such that for every $\gamma \in (0,\gamma_0)$ the unadjusted Langevin algorithm (ULA) given in \eqref{eq: ULA} satisfies
\begin{equation}
W_p(\pi_\beta, \mathcal{L}(x_n)) \leq W_p(\xi, \pi_\beta)e^{-\mu \gamma n}+c\gamma^{1/4}.
\end{equation}
\end{thm}
\begin{rem}
We note that Theorem \ref{th: main theorem} replicates the numerical error of order $1/2$, established in the case where $\nabla U$ is Lipschitz in \cite{articleula} (see Proposition 3 and Corollary 7). A numerical error of order $1$ is also established in the same paper under the assumption $\nabla U$ is smooth (see Corollary 9). However, it is shown in \cite{10.1093/imanum/draa078} that in general the Euler scheme cannot converge faster than rate $3/4$ for discontinuous drifts, which suggests that Theorem \ref{th: main theorem} cannot be improved beyond numerical error of order $3/4$.
\end{rem}
\begin{rem}
For the Lipschitz case, letting $L>0$ denote the Lipschitz constant, the stepsize restriction is $\gamma \leq \frac{1}{\mu+L}$ (see \cite{articleula}). This is comparable to our stepsize restriction in the case $\mu \sim L$, as in the case for instance where $\nabla U(x)$ is a perturbation of $\mu x$, so that $\pi_\beta$ is a perturbation of a Gaussian. However as $L\to\infty$ for fixed $\mu$, the  stepsize restriction in Theorems \ref{th: main theorem} and \ref{th: main theorem 2} scales like $O(L^{-2})$ compare to $O(L^{-1})$ in \cite{articleula}.
\end{rem}
\begin{rem}
    Since the arguments in Section \ref{sec: crossing bounds} are significantly intricate, we choose not to track any dependence besides the step-size $\gamma>0$. However, one can nonetheless make the important observation that the bound in Theorem $\ref{th: main theorem}$ blows up as $\beta\to\infty$. This arises due to \eqref{eq: local time t} and \eqref{eq: local time t 2}, which cause the bound in Lemma \ref{lem: time dep local time} to scale with $\beta$, and therefore (due to \eqref{eq: v1}) causes the bound in Lemma \ref{lemma: cross bdd p=1 1} to scale similarly. Therefore, our results can be seen as an example of `regularisation by noise', since the noise is essential for avoiding pathological behaviour, and our arguments become less effective as the process becomes more deterministic. This is in contrast to Theorem \ref{th: main theorem 2}, where the bound will strengthen as $\beta\to\infty$, since the error in Theorem \ref{th: main theorem 2} depends on the increment bound in Lemma \ref{lemma: incr bound}, which will decrease and in fact be of order $O(l^p)$ in the $\beta\to\infty$ limit.
\end{rem}

Although we have presented Theorems \ref{th: main theorem} and \ref{th: main theorem 2} in terms of Wasserstein distance (as is standard in the algorithms literature), we prove the theorem by obtaining bounds in $L^p$ of the kind common in the numerics literature, see  \eqref{eq: uniform numerics bound}. In particular, we show that that if $Y_0=X_0=\xi$, then under the hypothesis of Theorem \ref{th: main theorem} one has
\begin{equation}
\sup_{t\geq 0} \; (E[\lvert Y_t-X_t\rvert^p])^{1/p}\leq c\gamma^{1/2},
\end{equation}
and similarly under the hypothesis of Theorem \ref{th: main theorem 2} one has
\begin{equation}
\sup_{t\geq 0}\; ( E[\lvert Y_t-X_t\rvert^p])^{1/p}\leq c\gamma^{1/4}.
\end{equation}
\section{Preliminary Bounds}\label{sec: prelim bounds}
\begin{note}
From now on the generic constant $c>0$ changes from line to line and is \textbf{independent of the stepsize $\gamma>0$ and time parameters $t,s,u>0$}.
\end{note}
\begin{note}\label{note: youngs ineq}
We use often the following corollary of Young's inequality: for every $p,q>0$ and $a>0$ there exists $c>0$ such that $x^py^q\leq ax^{p+q}+cy^{p+q}$. This follows by applying Young's inequality to $x^py^q=(a^{\frac{p}{p+q}}x^p)(a^{-\frac{p}{p+q}}y^q)$.
\end{note}
In this section we prove standard moment and increment bounds for the Langevin dynamics \eqref{eq: main SDE}, the Euler scheme \eqref{eq: ULA} and and its continuous interpolation \eqref{eq: cont interp}. These bounds are very similar to those in \cite{articleula} and other references, with minor alterations to account for the lack of continuity of $\nabla U$. Recall that \Cref{assmp: piecewise Lip} implies \Cref{assmp: growth assumption}.

All Lemmas in Sections \ref{sec: prelim bounds} and \ref{sec: crossing bounds} are proven under the expectation $E^x$ defined in Section \ref{subsec: setup} (which implies the same bound under $E$ with the use of \Cref{assmp: initial condition}). This is due to the proof of Proposition \ref{prop: exponential crossing}, where the dependence on the initial condition (and the fact that the $\Sigma_j$ lie in a compact set) is used at the conclusion of the proof.
\begin{lemma}\label{lemma: moment bounds for cont}
Let \Cref{assmp: growth assumption} and \Cref{assmp: strong mon} hold. Let $p>0$ and $x \in \mathbb{R}^d \setminus \cup_{j=1}^{n_2}\Sigma_j$. Then one has
\begin{equation}
\sup_{t\geq 0} E^x \lvert Y_t\rvert^p \leq c(1+\lvert x \rvert^p).
\end{equation}
\end{lemma}
\begin{proof}
Let $\tau_R:=\inf \{t\geq 0 \vert \; \lvert Y_t \rvert =R  \}$. Then one may begin by applying It\^{o}'s formula on $[0,\tau_R \wedge t]$, so that by the boundedness of the coefficients under \Cref{assmp: growth assumption}, the stochastic integral vanishes and one obtains
\begin{align}
E^xe^{p\mu (\tau_R \wedge t)/4}\lvert Y_{\tau_R \wedge t}\rvert^p&=\lvert x\rvert^p+E^x\int^{\tau_R \wedge t}_0 e^{p\mu s/4}(\frac{p\mu}{4}\lvert Y_s \rvert^p-p\langle \nabla U(Y_s), Y_s\rangle \lvert Y_s\rvert^{p-2}+ c\lvert Y_s\rvert^{p-2})ds,
\end{align}
Now let us assume without loss of generality that $\nabla U(0)$ is well defined (since $\nabla U$ is well defined almost everywhere, this holds up to a slight shift of the coordinate system). Then we may use \Cref{assmp: strong mon} to write
\begin{align}\label{eq: coerc}
-\langle \nabla U(x), x\rangle &\leq -\langle \nabla U(x)-\nabla U(0), x\rangle -\langle \nabla U(0), x\rangle \nonumber \\
&\leq -\mu\lvert x \rvert^2+\frac{\mu}{2}\lvert x \rvert^2+\frac{2}{\mu} \lvert \nabla U(0)\rvert^2 \nonumber \\
&\leq -\frac{\mu}{2}\lvert x \rvert^2+c,
\end{align}
so that applying this bound, and using Young's inequality to bound
 the last term (see Note \ref{note: youngs ineq}) as
 \begin{equation}
    c\lvert Y_s\rvert^{p-2}\leq \frac{p\mu}{4}\lvert Y_s\rvert^p+c,
 \end{equation}
 one has
\begin{align}\label{eq: stopped bound}
E^xe^{p\mu (\tau_R \wedge t)/4}\lvert Y_{\tau_R \wedge t} \rvert^p&=\lvert x\rvert^p+c\int^{\tau_R \wedge t}_0 e^{p\mu s/4}ds \leq \lvert x\rvert^p+ce^{p\mu (\tau_R \wedge t)/4} \leq \lvert x\rvert^p+ce^{p\mu t/4}.
\end{align}
Now observe that by continuity $\sup_{s\in [0,t]}\lvert Y_s\rvert <\infty$ almost surely, so that by Fatou's Lemma
\begin{equation}
\liminf_{R\to\infty} E^x e^{p\mu (\tau_R \wedge t)/4}\lvert Y_{\tau_R \wedge t} \rvert^p \geq E^x \liminf_{R\to\infty}  e^{p\mu (\tau_R \wedge t)/4}\lvert Y_{\tau_R \wedge t} \rvert^p= e^{p\mu t/4}E^x \lvert Y_t \rvert^p.
\end{equation}
Therefore since the RHS of \eqref{eq: stopped bound} is independent of $R$, we may take $\liminf_{R\to\infty}$ and divide through by $e^{p\mu t/4}$, so that
\begin{equation}
E^x \lvert Y_t \rvert^p \leq e^{-p\mu t/4}\lvert x\rvert^p+c\leq \lvert x\rvert^p+c,
\end{equation}
at which point the result follows since this bound is independent of $t\geq 0$.
\end{proof}

\begin{lemma}\label{lemma: moment bounds for scheme}
Let \Cref{assmp: growth assumption} and \Cref{assmp: strong mon} hold. Let $p>0$.  Let $\gamma_0 \in (0, \frac{\mu}{2L^2})$. Then there exists $c>0$ such that for every $\gamma \in (0,\gamma_0)$ and $x \in \mathbb{R}^d \setminus \cup_{j=1}^{n_2}\Sigma_j$ one has
\begin{equation}
\sup_{t\geq 0} E^x \lvert X_t\rvert^p \leq c(1+\lvert x \rvert^p).
\end{equation}
\end{lemma}
\begin{proof}
Let us fix $\gamma_0 \in (0, \frac{\mu}{2L^2})$ and consider $X_t$ for stepsize $\gamma \in (0,\gamma_0)$. Recalling the discrete definition in \eqref{eq: ULA}, we begin by proving
\begin{equation}\label{eq: disc moment bounds p=2}
 \sup_{n\geq 0} E^x \lvert x_n\rvert^2 \leq c(1+\lvert x \rvert^2).
\end{equation}
Let us begin by applying \Cref{assmp: growth assumption} and \eqref{eq: coerc} and calculating
    \begin{align}
     \lvert x_n-\gamma \nabla U(x_n)  \rvert^2&= \lvert x_n \rvert^2+  \gamma^2 \lvert \nabla U(x_n)\rvert^2-2\gamma\langle \nabla U(x_n), x_n \rangle  \nonumber \\
     & \leq r(\gamma)\lvert x_n\rvert^2+c\gamma, \nonumber
    \end{align}
    where $r(\gamma) := \max\{ 1-\gamma \mu+2L^2\gamma^2, 0\}$, since if $1-\gamma \mu+2L^2\gamma^2\leq 0$,  we may bound by $0$. Then one has by the definition of \eqref{eq: ULA} that
    \begin{equation}\label{eq: inductive expression}
    \lvert x_{n+1}  \rvert^2 \leq r(\gamma)\lvert x_n \rvert^2+c\gamma+ \frac{2\gamma}{\beta}\lvert z_{n+1} \rvert^2+2\langle   x_n -\gamma \nabla U(x_n) , \sqrt{\frac{2\gamma}{\beta}} z_{n+1}\rangle 
    \end{equation}
    so that denoting $\eta_{i+1}:=2\sqrt{\gamma} \langle   x_n -\gamma \nabla U(x_n) , \sqrt{\frac{2}{\beta}}z_{i+1}\rangle$ one may use the standard identity that, for sequences $a_n, b_n \in \mathbb{R}$ and $r>0$ satisfying $a_{n+1}\leq ra_n+b_n$, one has
    \begin{equation}
        a_n\leq r^na_0+\sum_{i=0}^{n-1} r^i b_i,
    \end{equation}
    to calculate that
\begin{equation}\label{eq: p=2 moment bound}
\lvert x_n  \rvert^2\leq r(\gamma)^n\lvert x_0\rvert^2+c\gamma  \sum^{n-1}_{i=0}r(\gamma)^i+\sum^{n-1}_{i=0}(\frac{2\gamma}{\beta}\lvert z_{i+1} \rvert^2+\eta_{i+1} )r(\gamma)^i.
\end{equation}
Then since $x_n$ is independent of $z_{n+1}$, one has $E^x\eta_i=0$ for every $i\geq 1$. As a result, one has that
\begin{equation}
E^x\lvert x_n  \rvert^2\leq \lvert x\rvert^2+c\frac{1}{1-r(\gamma)} \leq \lvert x\rvert^2+c\frac{1}{2\mu-4L^2\gamma}.
\end{equation}
Note that since $\gamma_0 \in (0, \frac{\mu}{2L^2})$, these bounds are uniform over $\gamma \in (0,\gamma_0)$ and $n\in \mathbb{N}$, so \eqref{eq: disc moment bounds p=2} follows. Now let us assume 
\begin{equation}\label{eq: disc moment bounds}
 \sup_{n\geq 0} E^x \lvert x_n\rvert^p \leq c(1+\lvert x \rvert^p),
\end{equation}
holds for $p=2q$ for $q\in \mathbb{N}$. Let us prove by induction it therefore holds for $p=2q+2$, thus proving \eqref{eq: disc moment bounds} for $p\in 2\mathbb{N}$, at which point the result follows for all $p>0$, by Jensen's inequality. To this end let us raise \eqref{eq: inductive expression} to the power of $q+1$, expand into polynomial factors and take expectation, so that by \Cref{assmp: growth assumption} and the inductive moments bound assumption
    \begin{equation}
    E^x \lvert x_{n+1}  \rvert^{2q+2} \leq r(\gamma)^{q+1} E^x \lvert x_n \rvert^{2q+2}+c\gamma,
    \end{equation}
since all terms of lower order in $\gamma$ vanish due to the independence of $z_{n+1}$ with $x_n$. Therefore
    \begin{equation}
    E^x \lvert x_{n+1}  \rvert^{2q+2} \leq r(\gamma)^{n(q+1)} \lvert x \rvert^{2q+2} +c\gamma \frac{1}{1-r(\gamma)^{q+1}},
    \end{equation}
and so since $r(\gamma)^{q/2+1} \leq r(\gamma) <1$ one has
    \begin{equation}
    E^x \lvert x_{n+1}  \rvert^{2q+2} \leq r(\gamma)^{n(q+1)} \lvert x \rvert^{2q+2} +c\gamma \frac{1}{1-r(\gamma)} \leq \lvert x \rvert^{2q+2}+\frac{1}{2\mu-4L^2\gamma},
    \end{equation}
and \eqref{eq: disc moment bounds} follows uniformly over $\gamma \in (0, \gamma_0)$ for $p=2q+2$, as required. Now for the full result observe that by H\"{o}lder's inequality
\begin{align}
\lvert X_t\rvert^p&=c\lvert X_t-X_{\kappa_\gamma(t)}\rvert^p+c\lvert X_{\kappa_\gamma(t)}\rvert^p \nonumber \\
&\leq c\gamma^{p-1}\int^t_{\kappa_\gamma(t)} \lvert \nabla U(X_{\kappa_\gamma(t)})\rvert^{p}ds +c\lvert W_t-W_{\kappa_\gamma(t)}\rvert^p+c\lvert X_{\kappa_\gamma(t)}\rvert^p,
\end{align}
at which point one applies \Cref{assmp: growth assumption} to obtain
\begin{align}
\lvert X_t\rvert^p\leq c\lvert X_{\kappa_\gamma(t)}\rvert^p+ c\lvert W_t-W_{\kappa_\gamma(t)}\rvert^p ,
\end{align}
and the result follows by applying expectation, since $X_{\kappa_\gamma(t)}=x_{\kappa_\gamma(t)/\gamma}$.
\end{proof}

\begin{lemma}\label{lemma: incr bound}
Let \Cref{assmp: growth assumption} and \Cref{assmp: strong mon} hold. Let $p,s>0$ and $l\in [0,1]$. Let $\gamma_0 \in (0, \frac{\mu}{2L^2})$. Then there exists $c>0$ such that for every $\gamma\in (0,\gamma_0)$ and $x \in \mathbb{R}^d \setminus \cup_{j=1}^{n_2}\Sigma_j$
\begin{equation}
E^x \sup_{u\in[s,s+l]}\lvert X_u-X_s\rvert^p\leq cl^{p/2}(1+\lvert x \rvert^p).
\end{equation}
\end{lemma}
\begin{proof}
One has
\begin{equation}
\lvert  X_u-X_s \rvert \leq \int^u_s \lvert \nabla U(X_{\kappa_\gamma(s)})\rvert ds +\sqrt{\frac{2\gamma}{\beta}} \lvert W_u-W_s\rvert,
\end{equation}
so that taking supremum and raising to the power $p$, by H\"{o}lder's inequality one has
\begin{equation}
\sup_{u\in[s,s+l]}\lvert  X_u-X_s \rvert ^p \leq cl^{p-1} \int^{s+l}_s \lvert \nabla U(X_{\kappa_\gamma(s)})\rvert^p ds +c\biggr(\frac{2\gamma}{\beta}\biggr)^{p/2} \sup_{u\in[s,s+l]} \lvert W_u-W_s\rvert^p,
\end{equation}
and therefore taking expectation and applying \Cref{assmp: growth assumption} and Lemma \ref{lemma: moment bounds for scheme}, the result follows.
\end{proof}

\begin{lemma}\label{lemma: ergodic bound}
Let \Cref{assmp: growth assumption} and \Cref{assmp: strong mon} hold. Let \eqref{eq: main SDE} have starting condition $Y_0=\xi$. Then
\begin{equation}
W_p(\pi_\beta, \mathcal{L}(Y_t))\leq W_p(\xi, \pi_\beta)e^{-\mu \gamma n}.
\end{equation}
\end{lemma}
\begin{proof}
Let $(\tilde{Y}_t)_{t\geq0}$ be a solution of \eqref{eq: main SDE}, driven by the same noise as $(Y_t)_{t\geq 0}$ but with initial condition $\tilde{Y}_0 = \tilde{\xi}$ satisfying $\mathcal{L}(\tilde{\xi})=\pi_\beta$. Then since $\pi_\beta$ is the invariant measure of \eqref{eq: main SDE}, for all $t\geq 0$ one has $\mathcal{L}(\tilde{Y}_t)=\pi_\beta$. One then calculates by the chain rule (since $Y_t-\tilde{Y_t} $ has vanishing diffusion) that
\begin{align}
e^{p\mu t} E\lvert Y_t-\tilde{Y_t} \rvert^p &= E\lvert \xi - \tilde{\xi}\rvert^p +\int^t_0 p\mu e^{p\mu s}\lvert Y_s-\tilde{Y_s} \rvert^p ds \nonumber \\
&+p\int^t_0 e^{p\mu s} \langle Y_s-\tilde{Y_s} , \nabla U(Y_s)-\nabla U(\tilde{Y_s}) \rangle \lvert  Y_s-\tilde{Y_s} \rvert^{p-2} ds,
\end{align}
so that applying the convexity assumption \Cref{assmp: strong mon} and dividing through by $e^{p\mu t}$
\begin{equation}
E\lvert Y_t-\tilde{Y_t} \rvert^p \leq e^{-p\mu t} E\lvert \xi - \tilde{\xi}\rvert^p.
\end{equation}
Then the result follows since one has the trivial bound $W_p(\pi_\beta, \mathcal{L}(Y_t))^p\leq E\lvert Y_t-\tilde{Y_t} \rvert^p$, and by optimising the coupling of $\xi$ and $\tilde{\xi}$ by the definition of the Wasserstein distance.
\end{proof}

\section{Crossing Bounds}\label{sec: crossing bounds}
Since $\nabla U$ is not continuous, under \Cref{assmp: piecewise Lip} we can only bound the crucial discretisation term $\lvert \nabla U(X_u)-\nabla U(X_{\kappa_\gamma(u)}) \rvert$ when $X_u$ and $X_{\kappa_\gamma(u)}$ both lie in the same $\Phi_i$ for some $i=1,2,...,n_1$. In this section we prove that this is the case a large amount of the time on an unbounded interval. In particular, in Proposition \ref{prop: exponential crossing} we show that the $p$th moment of the $du$-measure of the time $X_u$ and $X_{\kappa_\gamma(u)}$ spend in \textit{different} of the $\Phi_i$ on $[0,t]$, weighted by $e^{qu}$, is $O(\gamma^{p/2}e^{pqt})$. The exponential weighting here follows from the proof of Theorem \ref{th: main theorem}, where one scales the discretisation error by an appropriate exponential in order to prove bounds that are uniform in time.

The strategy of this section is as follows. In Lemmas \ref{lemma: local time dif} and \ref{lem: time dep local time} we show how you can use the classical local time identity to yield a local time identity for functions of the process integrated on $[t,t+1]$ for $t\geq 0$. In Lemma \ref{lemma: cross bdd p=1 1} we use this to bound the times in which $X_u$ and $X_{\kappa_\gamma(u)}$ lie in different $\Phi_i$, for $u \in [t,t+1]$. In Lemma \ref{lemma: cross bdd p=1 2} we extend the previous lemma to the unbounded interval $[0,t]$, weighted by an exponential. Finally in Proposition \ref{prop: exponential crossing} we prove the full result.

\begin{lemma}\label{lemma: local time dif}
Let \Cref{assmp: piecewise Lip} and \Cref{assmp: strong mon} hold. Let $\gamma_0 \in (0, \frac{\mu}{2L^2})$. Let $P_j$ be the local signed distance function to $\Sigma_j$ given in Section \ref{subsec: delta nbhd} for $j=1,2,...,n_2$. Let $\mathcal{L}^{a,j}_t$ denote the local time of $P_j(X)$ at $a \in \mathbb{R}$ and time $t\geq 0$. Then there exists $c>0$ such that for every  $t\geq 0$, $\gamma \in (0,\gamma_0)$, $a\in [-\delta, \delta]$, $j \in \{1,2,...,n_2\}$ and $x \in \mathbb{R}^d \setminus \cup_{j=1}^{n_2}\Sigma_j$ one has 
\begin{equation}\label{eq: l time lemma 1}
E^x\mathcal{L}^{a,j}_t \leq c(1+\lvert x \rvert)(1+t),
\end{equation}
\begin{equation}\label{eq: l time lemma 2}
E^x(\mathcal{L}^{a,j}_{t+1}-\mathcal{L}^{a,j}_t )\leq c(1+\lvert x \rvert).
\end{equation}
\end{lemma}
\begin{proof}
Since $P_j \in C^2(\mathbb{R}^d)$ one may apply the classical It\^{o}'s formula to obtain
\begin{equation}
P_j(X_t)= P_j(\xi)+\int^t_0 \biggr (-\nabla P_j(X_u)\nabla U(X_{\kappa_\gamma(u)})+\frac{1}{\beta} \Delta P_j(X_{\kappa_\gamma(u)}) \biggr) du+\sqrt{\frac{2}{\beta}}\int^t_0 \nabla P_j(X_u) dW_u
\end{equation}
Therefore by the classical local time identity, or Tanaka-Meyer identity (see (7.9) on page 220 in \cite{karatzas1991brownian}), one has
\begin{align}\label{eq: classical local time}
E^x \mathcal{L}^{a,j}_t&=\frac{1}{2}E^x \lvert P_j(X_t)-a \rvert -  \frac{1}{2}\lvert P_j(x)-a \rvert-\frac{1}{2}E^x \int^t_0 sign(P_j(X_s)-a)\nabla P_j(X_s) \nabla U(X_{\kappa_\gamma(s)}) ds \nonumber \\
&-\frac{1}{2\beta}E^x \int^t_0 sign(P_j(X_s)-a) \Delta P_j(X_s)  ds,
\end{align}
Now observe that since $P_j \in C^2(\mathbb{R}^d)$ and is supported on a compact set, $P_j$, $\nabla P_j$ and $\nabla^2 P_j$ are bounded. Therefore 
\begin{align}
E^x \mathcal{L}^{a,j}_t&\leq \frac{1}{2} E^x \lvert P_j(X_t)-a \rvert -  \frac{1}{2} \lvert P_j(x)-a \rvert+c \int^t_0 E^x\lvert \nabla U(X_{\kappa_\gamma(s)}) \rvert ds \nonumber \\
&-\frac{1}{2\beta}E^x \int^t_0 sign(P_j(X_s)-a) \Delta P_j(X_s)  ds,
\end{align}
so by the triangle inequality, \Cref{assmp: growth assumption} (since \Cref{assmp: piecewise Lip} implies \Cref{assmp: growth assumption}) and Lemma \ref{lemma: moment bounds for scheme}, one can ensure all terms are above bounded by $c(1+\lvert x \rvert)(1+t)$, and \eqref{eq: l time lemma 1} follows. For the second bound first observe that by the triangle inequality
\begin{equation}\label{eq: Pi dif}
\lvert P_j(X_{t+1})-a \rvert -\lvert P_j(X_t)-a \rvert \leq \lvert P_j(X_{t+1})-P_j(X_t) \rvert ,
\end{equation}
so that by \eqref{eq: classical local time} one has that the difference $E^x \mathcal{L}^{a,j}_{t+1}-E^x \mathcal{L}^{a,j}_t$ is given by integrals over $[t,t+1]$ and \eqref{eq: Pi dif}, so one may obtain
\begin{align}
E^x \mathcal{L}^{a,j}_{t+1}-E^x \mathcal{L}^{a,j}_t&\leq E^x \lvert P_j(X_{t+1})-P_j(X_t) \rvert +c.
\end{align}
Now observe that additionally $ P_j$ is Lipschitz since it has uniformly bounded first derivative, at which point \eqref{eq: l time lemma 2} follows from Lemma \ref{lemma: incr bound} with $l=1$.
\end{proof} 
\begin{lemma}\label{lem: time dep local time}
Let \Cref{assmp: piecewise Lip} and \Cref{assmp: strong mon} hold. Let $\gamma_0 \in (0, \frac{\mu}{2L^2})$. Let $g: \mathbb{R}\to\mathbb{R}$ satisfy $0\leq g \leq c$ and suppose $supp(g)\subset (-\delta,\delta)$ for $\delta>0$ as in Proposition \ref{prop: good delta}. Then there exists $c>0$ such that for every $t\geq 0$, $\gamma \in (0,\gamma_0)$, $a\in [-\delta, \delta]$, $j \in \{1,2,...,n_2\}$ and $x \in \mathbb{R}^d \setminus \cup_{j=1}^{n_2}\Sigma_j$ one has
\begin{equation}
E^x\int^{t+1}_t g(P_j(X_s))ds\leq c(1+\lvert x \rvert)\int^\delta_{-\delta} g(s)ds.
\end{equation}
\end{lemma}
\begin{proof}
Let us denote the quadratic variation of $P_j(X)$ by $\langle P_j(X)\rangle_s$. Then since $supp(g)\subset (-\delta,\delta)$, and $P_j(x)\in [-\delta, \delta]$ implies $x\in \Sigma_j^\delta$, one has
\begin{equation}
\int^t_0 g(P_j(X_s))d\langle P_j(X)\rangle_s = \frac{2}{\beta}\int^t_0 g(P_j(X_s)) \lvert \nabla P_j(X_s)\rvert^2 ds = \frac{2}{\beta}\int^t_0 g(P_j(X_s)) \lvert \nabla P_j(X_s)\rvert^2 1_{X_s \in \Sigma_j^\delta}ds,
\end{equation}
so that since $\lvert \nabla \rho_j(x)\rvert = \lvert \nabla P_j(x)\rvert=1$ for every $x\in \Sigma_j^\delta$ by Proposition \ref{prop: good delta}, one has 
\begin{equation}
\int^t_0 g(P_j(X_s))  ds =\frac{\beta}{2}\int^t_0 g(P_j(X_s))d\langle P_j(X)\rangle_s .
\end{equation}
Therefore, by the classical local time identity, i.e. Theorem 7.1 iii) page 218 in \cite{karatzas1991brownian}, one obtains
\begin{equation}\label{eq: local time t}
E^x \int^t_0 g(P_j(X_s))ds=\frac{\beta}{2}  \int^t_0 E^x g(P_j(X_s))d\langle P_j(X)\rangle_s = \frac{\beta}{2}  \int^\delta_{-\delta} g(a)E^x\mathcal{L}^{a,j}_t da,
\end{equation}
so that
\begin{align}\label{eq: local time t 2}
E^x\int^{t+1}_t g(P_i(X_s))ds&=E^x\int^{t+1}_0 g(P_i(X_s))ds-E^x\int^t_0 g(P_i(X_s))ds\nonumber \\
&=\frac{\beta}{2}  \int^\delta_{-\delta} g(a)E^x(\mathcal{L}^{a,j}_{t+1}-\mathcal{L}^{a,j}_t) da,
\end{align}
and therefore the result follows from Lemma \ref{lemma: local time dif}. 
\end{proof}
\begin{lemma}\label{lemma: cross bdd p=1 1}
Let \Cref{assmp: piecewise Lip} and \Cref{assmp: strong mon} hold. Let $\gamma_0 \in (0, \frac{\mu}{2L^2})$. Then there exists $c>0$ such that for every $t\geq0$, $\gamma \in (0,\gamma_0)$ and $x \in \mathbb{R}^d \setminus \cup_{j=1}^{n_2}\Sigma_j$ one has
\begin{equation}
E^x \int^{t+1}_t  1_{A_u}du  \leq c (1+\lvert x \rvert^3) \gamma^{1/2} , \nonumber
\end{equation}
where 
\begin{equation}\label{eq: A def}
A_s:=\{ \text{There does not exist} \; i\in \{1,2,...,m\} \;\text{such that}\; X_s, X_{\kappa_\gamma(s)}\in\Phi_i \}. 
\end{equation}
\end{lemma}
\begin{proof}
\textit{Step i)}
Let us fix a time $t\geq0$ and $j\in \{1,2,...,n_2\}$. Then we may define a stopped version of $X_s$ which is equal to $X_s$ unless there happens to be a $u\in[t,t+1]$ for which $\lvert X_u-X_{\kappa_\gamma(u)} \rvert \geq 1$, at the first instance of which it stops. Specifically
\begin{equation}\label{eq: stopping defn}
\tau:=\inf\{ u\in [t,t+1+\gamma]  \; \vert \; \lvert X_u-X_{\kappa_\gamma(u)} \rvert \geq 1 \} \cup\{\infty\},\;\;\; \bar{X}_s:=X_{s \wedge \tau}, \;\;\; s\geq 0.
\end{equation}
Let us define the stopped version of $K^j_s$ from the next step as
\begin{equation}
    \bar{K}^j_s:= \biggr \{ \sup_{u \in [s,s+\gamma]} \lvert \bar{X}_u-\bar{X}_s \rvert \geq dist(\bar{X}_s, \Sigma_j) \biggr \}, \;\;\; s\geq 0,
\end{equation}
and (recalling $\tau\geq t$ by definition) let us prove
\begin{equation}\label{eq: stopped int}
E^x \int^{(t+1)\wedge \tau}_{t}  1_{\bar{K}^j_s}ds  \leq c  (1+\lvert x \rvert) \gamma^{1/2} .
\end{equation}
To this end we may split as
\begin{align}\label{eq: splitting}
\int^{(t+1)\wedge \tau}_{t} 1_{\bar{K}^j_s}ds  &= \int^{(t+1)\wedge \tau}_{t} 1_{\bar{K}^j_s}\cdot 1_{\bar{X}_s \in \Sigma_j^\delta} ds+ \int^{(t+1)\wedge \tau}_{t} 1_{\bar{K}^j_s}\cdot 1_{\bar{X}_s \not \in \Sigma_j^\delta} ds\nonumber \\
&=: v_1+v_2.
\end{align}
Now let us show that there exists $c_1>0$ such that for every $\omega \in \{ \bar{X}_s \in \Sigma_j^\delta \}$ and $s \in [t,t+1]$ one has
\begin{equation}\label{eq: stopped incr sup}
\sup_{u\in [s, s+\gamma]} \lvert \bar{X}_s(\omega)-\bar{X}_u(\omega)\rvert \leq c_1\biggr(\gamma+\sup_{u\in [s, s+\gamma]} \lvert W_s(\omega)-W_u(\omega)\rvert \biggr).
\end{equation} 
Let $s \in [t,t+1]$ and $l \in [s,s+\gamma]$. Then 
\begin{align}
\lvert \bar{X}_s-\bar{X}_l\rvert &\leq \biggr \lvert \int^{l\wedge \tau}_{s\wedge \tau}  \nabla U(\bar{X}_{\kappa_\gamma(v)})  dv+\sqrt{\frac{2}{\beta}}( W_{s\wedge \tau} - W_{l\wedge \tau}) \biggr \rvert\nonumber \\
& \leq \int^{s+\gamma}_{s} \lvert \nabla U(\bar{X}_{\kappa_\gamma(v)}) \rvert dv+\sqrt{\frac{2}{\beta}} \sup_{u\in [s,s+\gamma]}\lvert  W_s - W_u  \rvert,
\end{align}
so that since the bound is independent of $l$ one has
\begin{align}
\sup_{u\in [s,s+\gamma]} \lvert \bar{X}_s-\bar{X}_u\rvert  \leq \int^{s+\gamma}_{s} \lvert \nabla U(\bar{X}_{\kappa_\gamma(v)}) \rvert dv+\sqrt{\frac{2}{\beta}} \sup_{u\in [s,s+\gamma]}\lvert  W_s - W_u  \rvert.
\end{align}
Now note for $v \in [s,s+\gamma]$ one has $\kappa_\gamma(v)\in \{ \kappa_\gamma(s), \underline{\kappa}_\gamma(s) \}$, where $\underline{\kappa}_\gamma$ is the forward projection onto the discretisation grid given in Section \ref{subsec: setup}. Let us show we can bound the integrand above uniformly for $\omega \in \{ \bar{X}_s \in \Sigma_j^\delta \}$.
\begin{figure}[H]\centering
\includegraphics[width=8cm, height=0.9cm]{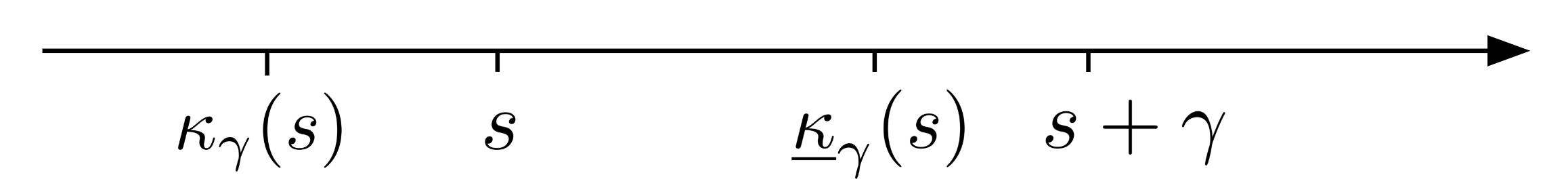}
\caption{Recall that the coefficients of \eqref{eq: cont interp} change at $\kappa_\gamma(s)$ and $\underline{\kappa}_\gamma(s)$}\label{visina8}
\end{figure}
First note that since $\Sigma_j^\delta \subset B_R$ for $R>0$ as in Section \ref{subsec: assmp}, one has that $\omega \in \{ \bar{X}_s \in \Sigma_j^\delta \}$ implies that $\lvert \bar{X}_s(\omega)\rvert \leq R$. Therefore, by the definition of the stopped process in \eqref{eq: stopping defn}, for such an $\omega \in \{ \bar{X}_s \in \Sigma_j^\delta \}$ one has 
\begin{equation}
    \lvert \bar{X}_{\kappa_\gamma(s)}(\omega)\rvert \leq \lvert \bar{X}_s(\omega)\rvert+\lvert \bar{X}_s(\omega)-\bar{X}_{\kappa_\gamma(s)}(\omega)\rvert\leq R+1,
\end{equation}
and
\begin{align}
    \lvert \bar{X}_{\underline{\kappa}_\gamma(s)}(\omega)\rvert &\leq \lvert \bar{X}_s(\omega)\rvert+\lvert \bar{X}_s(\omega)-\bar{X}_{\underline{\kappa}_\gamma(s)}(\omega)\rvert \nonumber \\
    &\leq \lvert \bar{X}_s(\omega)\rvert+\lvert \bar{X}_s(\omega)-\bar{X}_{\kappa_\gamma(s)}(\omega)\rvert+\lvert \bar{X}_{\kappa_\gamma(s)}(\omega)-\bar{X}_{\underline{\kappa}_\gamma(s)}(\omega)\rvert\leq R+2,
\end{align}
where the third term in the final line is bounded by $1$ by continuity, and since for $s\in [t,t+1]$ one has $\underline{\kappa}_\gamma(s)\leq t+1+\gamma$. Then since $\nabla U$ is bounded on compact sets, one has for $\omega \in \{ \bar{X}_s \in \Sigma_j^\delta \}$ that 
\begin{equation}
\int^{s+\gamma}_{s} \lvert \nabla U(\bar{X}_{\kappa_\gamma(v)}(\omega)) \rvert dv \leq c\gamma,
\end{equation}
at which point \eqref{eq: stopped incr sup} follows. Now let us denote (for all $\omega \in \Omega$)
\begin{equation}\label{eq: hs def}
h_s:=c_1\biggr(\gamma+\sup_{u\in [s, s+\gamma]} \lvert W_u-W_s\rvert \biggr).
\end{equation}
so that we have, by the definition of $\rho_j$ given in Section \ref{subsec: delta nbhd}
\begin{equation}
1_{K^j_s} \cdot 1_{\bar{X}_s \in \Sigma_j^\delta} \leq 1_{\{h_s \geq dist(\bar{X}_s, \Sigma_j^\delta)\}} \cdot 1_{\bar{X}_s \in \Sigma_j^\delta} = 1_{\{h_s \geq \lvert \rho_j (\bar{X}_s) \rvert\}} \cdot 1_{\bar{X}_s \in \Sigma_j^\delta}.
\end{equation}
As a result, since $\bar{X}_s=X_s$ for $s\in [0,\tau]$, one may write
\begin{align}
v_1& = \int^{(t+1)\wedge\tau}_{t}  1_{\bar{K}^j_s}\cdot 1_{\bar{X}_s \in \Sigma_j^\delta} ds \leq \int^{(t+1)\wedge\tau}_{t}  1_{\{h_s \geq \lvert \rho_j (\bar{X}_s) \rvert\}} \cdot 1_{\bar{X}_s \in \Sigma_j^\delta} ds \nonumber \\
&= \int^{(t+1)\wedge\tau}_{t} 1_{\{h_s \geq \lvert \rho_j (X_s) \rvert\}} \cdot 1_{X_s \in \Sigma_j^\delta} ds\leq \int^{t+1}_t 1_{\{h_s \geq \lvert \rho_j (X_s) \rvert\}} \cdot 1_{X_s \in \Sigma_j^\delta} ds.
\end{align}
Now recall that for independent random variables $X$ and $Y$ one has that
\begin{equation}
E[f(X,Y) \vert X]=g(X),
\end{equation}
where $g(x):=Ef(x,Y)$. Therefore, since $h_s$ is independent of $X_s$, one may write
\begin{align}
&E^x[1_{\{h_s \geq \lvert \rho_i (X_s) \rvert\}} \cdot 1_{X_s \in \Sigma_j^\delta}]=E^x\biggr[E^x [ 1_{\{h_s \geq \lvert \rho_i (X_s) \rvert\}} \cdot 1_{X_s \in \Sigma_j^\delta} \; \vert \;X_s]\biggr] \nonumber \\
&=E^x\biggr[E^x [ 1_{\{h_s \geq \lvert \rho_i (X_s) \rvert\}} \cdot 1_{\lvert \rho_i(X_s)\rvert\leq \delta} \; \vert \; \rho_i(X_s)] \biggr]=E^x[g(\rho_i(X_s))],
\end{align}
where since $\mathcal{L}(h_s)=\mathcal{L}(h_0)$ for $s\geq 0$, one may define $g:\mathbb{R}\to\mathbb{R}$ as
\begin{equation}
g(x)=P(\lvert x \rvert \leq h_0)1_{\lvert x \rvert \leq \delta} .
\end{equation}
Therefore, by the definition of $h_s$ in \eqref{eq: hs def}, one has $Eh_0=\leq c\gamma^{1/2}$, and by Lemma \ref{lem: time dep local time} one has
\begin{align}\label{eq: v1}
E^xv_1&\leq  E^x\int^{t+1}_t g(\rho_i(X_s)) ds \leq c(1+\lvert x \rvert) \int^\delta_{-\delta} g(s)ds \leq c(1+\lvert x \rvert)\int^\delta_0P(s\leq h_0) ds  \nonumber \\
&\leq c(1+\lvert x \rvert)\int^\infty_0  P(s \leq h_0)ds 
\leq c(1+\lvert x \rvert)Eh_0\leq c(1+\lvert x \rvert)\gamma^{1/2}.
\end{align}
Now for $v_2$ we observe that 
\begin{equation}
K^j_s \cap \{\bar{X}_s \not \in \Sigma_j^\delta \} \subset \{ \sup_{u \in [s,s+\gamma]} \lvert \bar{X}_u-\bar{X}_s \rvert \geq \delta  \},
\end{equation}
so by Markov's inequality and Lemma \ref{lemma: incr bound}
\begin{equation}\label{eq: v2}
E^x v_2\leq \int^{t+1}_t P^x(\sup_{u \in [s,s+\gamma]} \lvert \bar{X}_u-\bar{X}_s \rvert \geq \delta) ds\leq c(1+\lvert x \rvert)\gamma^{1/2},
\end{equation}
at which point substituting \eqref{eq: v1} and \eqref{eq: v2} into \eqref{eq: splitting}, the result in \eqref{eq: stopped int} follows.

\textit{Step ii)} Now we prove the unstopped version of \eqref{eq: stopped int}, that is for $t\geq 0$, $j\in \{1,2,...,n_2\}$ and $K^j_s$ given as
\begin{equation}
    K^j_s:= \biggr \{ \sup_{u \in [s,s+\gamma]} \lvert X_u-X_s \rvert \geq dist(X_s, \Sigma_j) \biggr \}.
\end{equation}
we prove
\begin{equation}\label{eq: unstopped int}
E^x \int^{t+1}_t  1_{K^j_u}du  \leq c  (1+\lvert x \rvert) \gamma^{1/2} .
\end{equation}
In the case where $\tau = \infty$, one has
\begin{equation}
    1_{\tau = \infty}\cdot \int^{t+1}_t  1_{K^j_s}ds \leq \int^{(t+1)\wedge \tau}_{t}  1_{\bar{K}^j_s}ds,
\end{equation}
so one may use \eqref{eq: stopped int} to bound as
\begin{equation}\label{eq: int tau large}
E^x \biggr [ 1_{\tau = \infty}\cdot \int^{t+1}_t  1_{K^j_s}ds  \biggr ]   \leq c(1+\lvert x \rvert)\gamma^{1/2}.
\end{equation}
For the case where $\tau\in [t,t+1]$, one may bound $1_{K^j_u}\leq 1$ so that
\begin{equation}\label{eq: int tau small 1}
E^x \biggr [1_{\tau\in [t,t+1]} \cdot \int^{t+1}_t 1_{K^j_u}du\biggr]    \leq P^x(\tau \in [t,t+1]).
\end{equation}
Now observing that $\tau \in [t,t+1]$ implies that there exists a $u\in [t,t+1]$ such that $\lvert X_u - X_{\kappa_\gamma(u)}\rvert \geq 1$, one has
\begin{equation}
\{ \tau \in [t,t+1]\} \subset  \{ \sup_{u\in [t, \underline{\kappa}_\gamma(t)]}\lvert  X_t -X_u\rvert \geq 1 \} \cup \biggr (\bigcup_{i=0}^{\lfloor 1/\gamma \rfloor} \{ \sup_{u\in [\underline{\kappa}_\gamma(t)+i\gamma, \underline{\kappa}_\gamma(t)+(i+1)\gamma]}\lvert X_{i\gamma}-X_u \rvert \geq 1 \} \biggr), \nonumber
\end{equation}
one may write
\begin{align}
P^x(\tau \in [t,t+1]) &\leq P^x(\sup_{u\in [t, \underline{\kappa}_\gamma(t)]}\lvert  X_t -X_u\rvert \geq 1)+\sum ^{\lfloor 1/\gamma \rfloor}_{i=0} P^x\biggr (\sup_{u\in [\underline{\kappa}_\gamma(t)+i\gamma, \underline{\kappa}_\gamma(t)+(i+1)\gamma]}\lvert X_{i\gamma}-X_u \rvert \geq 1 \biggr)\nonumber \\
&=P^x(\sup_{u\in [t, \underline{\kappa}_\gamma(t)]}\lvert  X_t -X_u\rvert ^3\geq 1)+\sum ^{\lfloor 1/\gamma \rfloor}_{i=0} P^x\biggr (\sup_{u\in [\underline{\kappa}_\gamma(t)+i\gamma, \underline{\kappa}_\gamma(t)+(i+1)\gamma]}\lvert X_{i\gamma}-X_u \rvert^3 \geq 1  \biggr). \nonumber
\end{align}
Then applying Markov's inequality and Lemma \ref{lemma: incr bound}, one obtains that
\begin{equation}
P^x(\tau \in [t,t+1] ) \leq c(\lfloor 1/\gamma \rfloor+1)(1+\lvert x \rvert^3)\gamma^{3/2}\leq c(1+\lvert x \rvert^3)\gamma^{1/2},
\end{equation}
which one may substitute into \eqref{eq: int tau small 1} to obtain
\begin{equation}\label{eq: int tau small 2}
E^x \biggr [1_{\tau \in [t,t+1]} \cdot \int^{t+1}_t  1_{K^j_s}ds\biggr]    \leq c(1+\lvert x \rvert^3)\gamma^{1/2},
\end{equation}
so that summing \eqref{eq: int tau large} and \eqref{eq: int tau small 2}, one has that \eqref{eq: unstopped int} follows.

\textit{Step iii)} Now we prove the result for $A_s$. Suppose $\omega\in (\cup_{j=1}^{n_2} K^j_{s-\gamma})^c$ for $s\geq \gamma$. Then $\sup_{u \in [s-\gamma,s]} \lvert X_u-X_s \rvert < dist(X_s, \Sigma_j)$ for every $j=1,,2,...,n_2$, so since $\cup_{i=1}^{n_1} \partial \Phi_i \subset \cup_{j=1}^{n_2} \Sigma_j$, the process can't have reached the boundary of any $\Phi_j$. Therefore it has stayed in the same $\Phi_i$, or in other words there must exist a $i'$ such that $X_u \in \Phi_{i'}$ for every $u\in [s-\gamma,s]$. Therefore $\omega\in(A_s)^c$ also. Negating this relation one obtains
\begin{equation}
A_s\subset \cup_{j=1}^{n_2}K^j_{s-\gamma}.
\end{equation}
Now let us assumme $\gamma \leq 1$ (if not the result follows trivially). Then we may bound $1_{A_u}$ by $1$ on $[t, t+\gamma]$ to obtain
\begin{equation}
E^x \int^{t+1}_t  1_{A_u}du= \gamma+ E^x \int^{t+1}_{t+\gamma}  1_{A_u}du\leq \gamma+\sum_{j=1}^{n_2} E^x \int^{t+1}_t  1_{K^j_u}du \leq c(1+\lvert x \rvert^3)\gamma^{1/2}
\end{equation}
by \eqref{eq: unstopped int}.
\end{proof}
\begin{lemma}\label{lemma: cross bdd p=1 2}
Let \Cref{assmp: piecewise Lip} and \Cref{assmp: strong mon} hold. Let $\gamma_0 \in (0, \frac{\mu}{2L^2})$. Then there exists $c>0$ such that for every $t\geq0$, $\gamma \in (0,\gamma_0)$ and $x \in \mathbb{R}^d \setminus \cup_{j=1}^{n_2}\Sigma_j$ one has
\begin{equation}
E^x \int^t_0 e^{qs} 1_{A_s}ds  \leq c (1+\lvert x \rvert^3) \gamma^{1/2} e^{q t}, \nonumber
\end{equation}
where $A_s$ is given in Lemma \ref{lemma: cross bdd p=1 1}.
\end{lemma}
\begin{proof}
Let us split $[0,t]$ into unit intervals up to $\lfloor t \rfloor$, and bound the exponential by its supremum on each interval. Then one obtains
\begin{equation}
\int^t_0 e^{qs} 1_{A_s}ds = \sum_{i=1}^{\lfloor t \rfloor }\int^i_{i-1} e^{qs} 1_{A_s} ds+\int^t_{\lfloor t \rfloor}e^{qs}1_{A_s} ds \leq \sum_{i=1}^{\lfloor t \rfloor }e^{qi} \int^i_{i-1}  1_{A_s} ds+e^{qt}\int^t_{\lfloor t \rfloor}1_{A_s} ds.
\end{equation}
Then by Lemma \ref{lemma: cross bdd p=1 1} one has 
\begin{align}
E^x \int^t_0 e^{qu} 1_{A_u}du &\leq \sum_{i=1}^{\lfloor t \rfloor }e^{qi} E^x\int^i_{i-1}  1_{A_s} ds+e^{qt}E^x\int^{\lfloor t \rfloor+1}_{\lfloor t \rfloor}1_{A_s} ds \nonumber \\
&\leq c(1+\lvert x \rvert^3)\gamma^{1/2} \biggr (\sum_{i=1}^{\lfloor t \rfloor} e^{qi}+e^{qt}\biggr) . \nonumber
\end{align}
Now observe that since $e^{qi}\leq e^{qu}$ for $u\in [i, i+1]$, one has
\begin{equation}
\sum_{i=1}^{\lfloor t \rfloor} e^{qi} \leq \sum_{i=1}^{\lfloor t \rfloor-1} \int^{i+1}_i e^{qu} du +e^{q\lfloor t \rfloor}\leq \int^{\lfloor t \rfloor } _1 e^{qu} du +e^{q\lfloor t \rfloor}\leq ce^{qt},
\end{equation}
at which point the result follows.
\end{proof}

\begin{prop}\label{prop: exponential crossing}
Let \Cref{assmp: piecewise Lip} and \Cref{assmp: strong mon} hold. Let $\gamma_0 \in (0, \frac{\mu}{2L^2})$ and $p\geq 0$. Then there exists $c>0$ such that for every $t\geq0$, $\gamma \in (0,\gamma_0)$ and $x \in \mathbb{R}^d \setminus \cup_{j=1}^{n_2}\Sigma_j$ one has
\begin{equation}
E\biggr [\int^t_0 e^{qs} 1_{A_s}ds \biggr ]^p \leq c  \gamma^{p/2} e^{pq t}, \nonumber
\end{equation}
where $A_s$ is given in Lemma \ref{lemma: cross bdd p=1 1}.
\end{prop}
\begin{proof}
We prove for $p\in \mathbb{N}$, at which point all other values of $p>0$ follow by Jensen's inequality and taking appropriate roots. Furthermore, we prove by induction, where the base case follows from Lemma \ref{lemma: cross bdd p=1 1}, the initial condition identity \eqref{eq: initial conditioin Fubini} and the integrability of the initial condition in Assumption \ref{assmp: initial condition}. Let us assume the Theorem holds for $p-1 \in \mathbb{N}$, and show that it therefore holds for $p \in \mathbb{N}$. Firstly, since the following integral is symmetrical over $s_1,...,s_p$, it is equal to the same integral multiplied by $p!$ with the condition that $s_1\leq s_2 \leq... \leq s_p$ (since the variables $s_p$ can be permuted in $p!$ ways). Using this trick one has
\begin{align} \label{eq: rewrite}
& \biggr (\int^t_0 e^{qs}  1_{A_s} ds \biggr)^p = \int^t_0... \int^t_0 e^{q(s_1+s_2+...+s_p)}1_{A_{s_1}} \cdot  1_{A_{s_2}} \cdot ... \cdot 1_{A_{s_p}} ds_1 ds_2 ... ds_p \nonumber \\
 & = p! \cdot \int_{0\leq s_1\leq ...\leq s_p \leq t} e^{q(s_1+s_2+...+s_p)} 1_{A_{s_1}} \cdot  1_{A_{s_2}} \cdot ... \cdot 1_{A_{s_p}} ds_1 ds_2 ... ds_p \nonumber \\
 & = p! \cdot  \int_{0\leq s_1\leq ...\leq s_{p-1} \leq t} \biggr( e^{q(s_1+s_2+...+s_{p-1})}1_{A_{s_1}} \cdot  1_{A_{s_2}} \cdot ... \cdot 1_{A_{s_{p-1}}}  \cdot \int^t_{s_{p-1}} e^{qs_p} 1_{A_{s_p}}ds_p \biggr) ds_1 ds_2 ... ds_{p-1}.
 \end{align}
 Then if we show
 \begin{equation}\label{eq: induct bound}
     E\biggr [1_{A_{s_1}} \cdot  1_{A_{s_2}} \cdot ... \cdot 1_{A_{s_{p-1}}}  \cdot \int^t_{s_{p-1}} e^{qs_p} 1_{A_{s_p}}ds_p \biggr ]\leq cE1_{A_{s_1}} \cdot  1_{A_{s_2}} \cdot ... \cdot 1_{A_{s_{p-1}}} \gamma^{1/2}e^{qt}+c\gamma^{p/2}e^{qt},
 \end{equation}
 we can apply expectation to \eqref{eq: rewrite}, insert \eqref{eq: induct bound} and apply the inductive assumption to prove the result. First suppose $s_{p-1}=\kappa_\gamma(s_{p-1})$. By Proposition \ref{prop: well posedness} one has $P(X_{s_{p-1}} \in \cup_{i=1}^{n_2} \Sigma_i)=0$, so $P(X_{s_{p-1}} \in \cup_{i=1}^{n_1} \Phi_i)=1$. Therefore by definition one has $1_{A_{s_{p-1}}}=0$ almost surely, so \eqref{eq: induct bound} holds in this case. Now suppose $s_{p-1}>\kappa_\gamma(s_{p-1})$. Similarly as in the proof of Lemma \ref{lemma: cross bdd p=1 1}, now let us split into the case where $\lvert X_{s_{p-1}}- X_{\kappa_\gamma(s_{p-1})}\rvert>1$ and $\lvert X_{s_{p-1}}- X_{\kappa_\gamma(s_{p-1})}\rvert\leq 1$. For the former one uses the trivial bound $1_{A_{s_i}}\leq 1$ to bound
 \begin{align}
     T_1&:=1_{A_{s_1}} \cdot  1_{A_{s_2}} \cdot ... \cdot 1_{A_{s_{p-1}}}   \cdot 1_{\{\lvert X_{s_{p-1}}- X_{\kappa_\gamma(s_{p-1})}\rvert>1 \}}\cdot \int^t_{s_{p-1}} e^{qs_p} 1_{A_{s_p}}ds_p \nonumber \\
     &\leq   1_{ \{ \lvert X_{s_{p-1}}- X_{\kappa_\gamma(s_{p-1})}\rvert>1 \}} \cdot \int^t_{s_{p-1}} e^{qs_p} ds_p, \nonumber
 \end{align}
 so that taking expectation, one may evaluate the integral to obtain
  \begin{align}\label{eq: T1}
     ET_1\leq ce^{qt}P(\lvert X_{s_{p-1}}- X_{\kappa_\gamma(s_{p-1})}\rvert^p>1)\leq c\gamma^{p/2}e^{qt}, 
 \end{align}
 by Markov's inequality and Lemma \ref{lemma: incr bound}. Now recall $R>0$ from Section \ref{subsec: assmp} and observe that for $A_{s_{p-1}}$ to hold one requires that one of either $X_{s_{p-1}}$ or $X_{\kappa_\gamma(s_{p-1})}$ belongs to $B_R$, as only one of the $\Phi_i$ (which we assume is $\Phi_1$) is unbounded. Therefore 
 \begin{equation}
   A_{s_{p-1}} \cap  \{\lvert X_{s_{p-1}}- X_{\kappa_\gamma(s_{p-1})}\rvert\leq 1 \} \subset A_{s_{p-1}} \cap  \{\lvert X_{s_{p-1}}\rvert\leq R+1 \}, \nonumber
 \end{equation}
 and so 
  \begin{align}
     T_2&:=1_{A_{s_1}} \cdot  1_{A_{s_2}} \cdot ... \cdot 1_{A_{s_{p-1}}}   \cdot 1_{\{\lvert X_{s_{p-1}}- X_{\kappa_\gamma(s_{p-1})}\rvert\leq 1\}} \cdot \int^t_{s_{p-1}} e^{qs_p} 1_{A_{s_p}}ds_p\nonumber \\
     &\leq 1_{A_{s_1}} \cdot  1_{A_{s_2}} \cdot ... \cdot 1_{A_{s_{p-1}}}   \cdot 1_{\{\lvert X_{s_{p-1}}\rvert  \leq R+1\}}\cdot \int^t_{s_{p-1}} e^{qs_p} 1_{A_{s_p}}ds_p.\nonumber
 \end{align}
 Now, due to the Euler-Scheme structure of the process we must split again. Specifically, recalling the forward projection onto $\{0,\gamma, 2\gamma, ... \}$ given in Sectiom \ref{subsec: setup} as $\underline{\kappa}_\gamma(t)$, we split into the case where  $\lvert X_{\underline{\kappa}_\gamma(s_{p-1})}- X_{s_{p-1}}\rvert>1$ and where $\lvert X_{\underline{\kappa}_\gamma(s_{p-1})}- X_{s_{p-1}}\rvert\leq 1$. For the former one writes
\begin{align}
 T_{2,1}&:=1_{A_{s_1}} \cdot  1_{A_{s_2}} \cdot ... \cdot 1_{A_{s_{p-1}}}  \cdot 1_{\{\lvert X_{s_{p-1}}\rvert\leq R+N\}} \cdot 1_{ \{ \lvert X_{\underline{\kappa}_\gamma(s_{p-1})}- X_{s_{p-1}}\rvert>1 \}}\cdot \int^t_{s_{p-1}} e^{qs_p} 1_{A_{s_p}}ds_p  \nonumber \\
 &=1_{A_{s_1}} \cdot  1_{A_{s_2}} \cdot ... \cdot 1_{A_{s_{p-1}}}   \cdot 1_{\{\lvert X_{s_{p-1}}\rvert\leq R+1\}} \cdot 1_{ \{ \lvert X_{\underline{\kappa}_\gamma(s_{p-1})}- X_{s_{p-1}}\rvert^p>1 \}}\cdot \int^t_{s_{p-1}} e^{qs_p} 1_{A_{s_p}}ds_p
\end{align}
 which one controls in the same way as $T_1$ (by bounding each indicator function other than the last by $1$), to obtain
 \begin{equation}\label{eq: T21}
E T_{2,1}\leq c\gamma^{p/2}e^{qt}.
\end{equation}
Furthermore, one may use the triangle inequality to write
\begin{align}
 T_{2,2}&:=1_{A_{s_1}} \cdot  1_{A_{s_2}} \cdot ... \cdot 1_{A_{s_{p-1}}}   \cdot 1_{\{\lvert X_{s_{p-1}}\rvert\leq R+N\}} \cdot 1_{ \{ \lvert X_{\underline{\kappa}_\gamma(s_{p-1})}- X_{s_{p-1}}\rvert\leq 1 \}} \cdot \int^t_{s_{p-1}} e^{qs_p} 1_{A_{s_p}}ds_p \nonumber \\
 &\leq 1_{A_{s_1}} \cdot  1_{A_{s_2}} \cdot ... \cdot 1_{A_{s_{p-1}}}    \cdot 1_{\{\lvert  X_{\underline{\kappa}_\gamma(s_{p-1})}\rvert\leq R+2 \}}\cdot \int^t_{s_{p-1}} e^{qs_p} 1_{A_{s_p}}ds_p,
\end{align}
so that splitting the integral over $[s_{p-1},t]$ into integrals on $[s_{p-1}, \underline{\kappa}_\gamma(s_{p-1})]$ and $[\underline{\kappa}_\gamma(s_{p-1}),t]$, and bounding $1_{A_{s_p}}\leq 1$ for the former, one has that
\begin{align}
 T_{2,2}&\leq 1_{A_{s_1}} \cdot  1_{A_{s_2}} \cdot ... \cdot 1_{A_{s_{p-1}}}\cdot \gamma e^{qt}\nonumber \\
 &+1_{A_{s_1}} \cdot  1_{A_{s_2}} \cdot ... \cdot 1_{A_{s_{p-1}}}   \cdot 1_{\{\lvert  X_{\underline{\kappa}_\gamma(s_{p-1})}\rvert\leq R+2 \}}\cdot \int^t_{\underline{\kappa}_\gamma(s_{p-1})}  e^{qs_p} 1_{A_{s_p}}ds_p ,\nonumber
\end{align}
so that
\begin{align}
 E[&T_{2,2} \vert \mathcal{F}_{\underline{\kappa}_\gamma(s_{p-1})}  ]\nonumber \\
 &\leq 1_{A_{s_1}} \cdot  1_{A_{s_2}} \cdot ... \cdot 1_{A_{s_{p-1}}}\biggr(\gamma e^{qt}+ 1_{\{\lvert  X_{\underline{\kappa}_\gamma(s_{p-1})}\rvert\leq R+2 \}} \biggr)E\biggr [\int^t_{\underline{\kappa}_\gamma(s_{p-1})} e^{qs_p} 1_{A_{s_p}}ds_p\biggr \vert \mathcal{F}_{\underline{\kappa}_\gamma(s_{p-1})} \biggr ].\nonumber
\end{align}
Now observe that, by Lemma \ref{lemma: Markov property} and Lemma \ref{lemma: cross bdd p=1 2}, one may write
\begin{align}\label{eq: cond exp bdd}
E\biggr [\int^t_{\underline{\kappa}_\gamma(s_{p-1})} e^{qs_p} 1_{A_{s_p}}ds_p\biggr \vert \mathcal{F}_{\underline{\kappa}_\gamma(s_{p-1})} \biggr ]&=E^{X_{\underline{\kappa}_\gamma(s_{p-1})}}\biggr [\int^{t-\underline{\kappa}_\gamma(s_{p-1})}_0 e^{q(s_p+\underline{\kappa}_\gamma(s_{p-1}))} 1_{A_{s_p}}ds_p\biggr] \nonumber \\
& \leq c(1+\lvert X_{\underline{\kappa}_\gamma(s_{p-1})} \rvert^3)\gamma^{1/2}e^{qt},
\end{align}
and therefore, the product of \eqref{eq: cond exp bdd} with $1_{\{\lvert  x_{\underline{\kappa}_\gamma(s_{p-1})}\rvert\leq R+2N \}}$ is bounded by $c(1+(R+2N)^3)\gamma^{1/2}e^{qt}$, and so
\begin{equation}\label{eq: T22}
ET_{2,2}\leq cE1_{A_{s_1}} \cdot  1_{A_{s_2}} \cdot ... \cdot 1_{A_{s_{p-1}}} \gamma^{1/2}e^{qt}. 
\end{equation}
Then summing \eqref{eq: T1}, \eqref{eq: T21} and \eqref{eq: T22} one achieves \eqref{eq: induct bound}, and therefore the result follows.
\end{proof}
\section{Proofs of Theorems \ref{th: main theorem} and \ref{th: main theorem 2}}

\begin{proof}[Proof of Theorem \ref{th: main theorem}]
Let $(Y_t)_{t\geq0}$ be a solution to \eqref{eq: main SDE}, satisfying $Y_0=\xi$, and driven by the same noise as the continuous interpolation $(X_t)_{t\geq0}$ of \eqref{eq: ULA} given in \eqref{eq: cont interp}. Since $\mathcal{L}(x_n)=\mathcal{L}(X_{n\gamma})$, we may use Lemma \ref{lemma: ergodic bound} and the properties of the Wasserstein distance to write
\begin{align}
W_p(\pi_\beta, \mathcal{L}(x_n)) &\leq W_p(\pi_\beta, \mathcal{L}(Y_{n\gamma}))+W_p(\mathcal{L}(Y_{n\gamma}), \mathcal{L}(X_{n\gamma})) \nonumber \\
& \leq e^{-\mu n \gamma}+(E\lvert X_{n\gamma}-Y_{n\gamma}\rvert^p)^{1/p},
\end{align}
so that if we define 
\begin{equation}\label{eq: e_t def}
    e_t:=Y_t-X_t,
\end{equation}
then to prove the Theorem it is sufficient to show
\begin{equation}\label{eq: uniform numerics bound}
\sup_{t\geq 0}E\lvert e_t \rvert^p \leq c\gamma^{p/2}.
\end{equation}
Since the diffusion of $e_t$ vanishes, we may begin by applying the chain rule to obtain
\begin{align}\label{eq: thm1 splitting}
e^{\frac{p \mu}{2} t} \lvert e_t\rvert^p&=\int^t_0 \frac{p \mu}{2} e^{\frac{p \mu}{2} s} \lvert e_s\rvert^p ds+\int^t_0 pe^{\frac{p \mu}{2} s}\langle \nabla U(Y_s)-\nabla U(X_{\kappa_\gamma(s)}), e_s\rangle \lvert e_s\rvert^{p-2} ds \nonumber \\
& = \int^t_0 \frac{p \mu}{2} e^{\frac{p \mu}{2} s} \lvert e_s\rvert^p ds+\int^t_0 pe^{\frac{p \mu}{2} s}\langle \nabla U(Y_s)-\nabla U(X_s), e_s\rangle \lvert e_s\rvert^{p-2}ds  \nonumber \\
&+\int^t_0 pe^{\frac{p \mu}{2} s}\langle \nabla U(X_s)-\nabla U(X_{\kappa_\gamma(s)}), e_s\rangle \lvert e_s\rvert^{p-2}ds \nonumber \\
&=: \frac{p \mu}{2} \int^t_0  e^{\frac{p \mu}{2} s} \lvert e_s\rvert^p ds+r_1(t)+r_2(t).
\end{align}
One can then apply the convexity assumption \Cref{assmp: strong mon} to bound $r_1(t)$ as
\begin{equation}\label{eq: r1}
r_1(t)\leq -p \mu\int^t_0  e^{\frac{p \mu}{2} s} \lvert e_s\rvert^p ds.
\end{equation}
For $r_2(t)$, let $\psi\in C^\infty_c(\mathbb{R}^d)$ satisfy $\psi(x)=1$ for $x\in B_R$ (where $R>0$ is as given in Section \ref{subsec: assmp}). Then one may write $\nabla U=f_1+f_2$ for
\begin{equation}
f_1:=\psi \nabla U,\;\;\; f_2:=(1-\psi)\nabla U,
\end{equation}
so that $f_1$ is piecewise Lipschitz on the $\Phi_i$ due to \Cref{assmp: piecewise Lip}, and $f_2$ is Lipschitz (since it vanishes on $B_R$, and therefore $supp ( f_2) \subset \Phi_1$). Furthermore, since $\psi$ has compact support, $f_1$ is bounded. Then one uses this splitting and the event $A_s$ given in \eqref{eq: A def} to split as
\begin{align}
r_2(t)&=\int^t_0 pe^{\frac{p \mu}{2} s}  \langle f_1(X_{\kappa_\gamma(s)})-f_1(X_s), e_s\rangle \lvert e_s\rvert^{p-2} \cdot  1_{\Omega \setminus A_s} ds\nonumber \\
&+\int^t_0 pe^{\frac{p \mu}{2} s}   \langle f_1(X_{\kappa_\gamma(s)})-f_1(X_s), e_s\rangle \lvert e_s\rvert^{p-2} \cdot 1_{A_s}  ds \nonumber \\
&+\int^t_0 pe^{\frac{p \mu}{2} s}\langle f_2(X_{\kappa_\gamma(s)})-f_2(X_s), e_s\rangle \lvert e_s\rvert^{p-2} ds \nonumber \\
&=: r_{2,1}(t)+r_{2,2}(t)+r_{2,3}(t).
\end{align}
For $r_{2,1}(t)$, since $X_{\kappa_\gamma(s)}, X_s \in \Phi_i$ for some $i=1,2,...,n_1$ one may apply \Cref{assmp: piecewise Lip} and Young's inequality (see Note \ref{note: youngs ineq}) to bound as
\begin{equation}\label{eq: r21}
r_{2,1}(t)\leq c\int^t_0 e^{\frac{p \mu}{2} s} \lvert X_s-X_{\kappa_\gamma(s)}\rvert^pds+ \frac{p \mu}{2} \int^t_0 e^{\frac{p \mu}{2} s} \lvert e_s\rvert^p  ds.
\end{equation}
For $r_{2,2}(t)$, one uses the boundedness of $f_1$ to write
\begin{align}
r_{2,2}(t)\leq c \int^t_0 e^{\frac{p \mu}{2} s} \lvert e_s\rvert^{p-1}1_{A_s}  ds= c \int^t_0 (e^{\frac{\mu (p-1)}{2} s} \lvert e_s\rvert^{p-1}) (e^{\frac{\mu }{2} s} 1_{A_s})  ds
\end{align}
so that we may pull out the supremum and apply Young's inequality to obtain
\begin{align}\label{eq: r22}
r_{2,2}(t)\leq c \sup_{u\in [0,t]} e^{\frac{\mu (p-1)}{2} u} \lvert e_u\rvert^{p-1}\int^t_0 e^{\frac{\mu }{2} s} 1_{A_s} ds\leq \frac{1}{2} \sup_{u\in [0,t]} e^{\frac{p \mu}{2} u} \lvert e_u\rvert^p+c\biggr(\int^t_0 e^{\frac{\mu }{2} s} 1_{A_s} ds\biggr)^p.
\end{align}
Finally for $r_{2,3}(t)$, one uses Young's inequality and the fact $f_2$ is Lipschitz to bound as
\begin{align}\label{eq: r23}
r_{2,3}(t)&\leq c\int^t_0 e^{\frac{p \mu}{2} s}\lvert f_2(X_s)-f_2(X_{\kappa_\gamma(s)})\rvert^p+\frac{p \mu}{2} \int^t_0 e^{\frac{p \mu}{2} s} \lvert e_s\rvert^p ds \nonumber \\
&\leq c\int^t_0 e^{\frac{p \mu}{2} s}\lvert X_s-X_{\kappa_\gamma(s)}\rvert^p ds+\frac{p \mu}{2} \int^t_0 e^{\frac{p \mu}{2} s} \lvert e_s\rvert^p ds.
\end{align}
Therefore, substituting \eqref{eq: r1}, \eqref{eq: r21},\eqref{eq: r22} and \eqref{eq: r23} into \eqref{eq: thm1 splitting}, one sees that
\begin{equation}
e^{\frac{p \mu}{2} t} \lvert e_t\rvert^p \leq c\int^t_0 \lvert X_s-X_{\kappa_\gamma(s)}\rvert^pe^{\frac{p \mu}{2} s} ds+\frac{1}{2} \sup_{u\in [0,t]} e^{\frac{p \mu}{2} u} \lvert e_u\rvert^p+c\biggr(\int^t_0 e^{\frac{\mu }{2} s} 1_{A_s} ds\biggr)^p,
\end{equation}
so that since the RHS is increasing as a function of $t>0$, one may take the supremum of the LHS for $u\in [0,t]$, move over the second term on the RHS (and multiply by $2$) to obtain
\begin{equation}
\sup_{u\in[0,t]}e^{\frac{p \mu}{2} u} \lvert e_u\rvert^p \leq c\int^t_0 \lvert X_s-X_{\kappa_\gamma(s)}\rvert^pe^{\frac{p \mu}{2} s} ds+c \biggr(\int^t_0 e^{\frac{\mu }{2} s} 1_{A_s} ds\biggr)^p.
\end{equation}
so that applying expectation, Lemma \ref{lemma: incr bound} and Proposition \ref{prop: exponential crossing}, and integrating the first term, one has
\begin{equation}
e^{\frac{p \mu}{2} t} E\lvert e_t\rvert^p \leq E \biggr (\sup_{u\in[0,t]}e^{\frac{p \mu}{2} u} \lvert e_u\rvert^p \biggr) \leq c\gamma^{p/2}e^{\frac{p \mu}{2} t},
\end{equation}
so that \eqref{eq: uniform numerics bound} follows by dividing through by $e^{\frac{p \mu}{2} t} $.
\end{proof}

\begin{proof}[Proof of Theorem \ref{th: main theorem 2}]
We follow the strategy above, but this time prove
\begin{equation}\label{eq: uniform numerics bound 2}
\sup_{t\geq 0}E\lvert e_t \rvert^p \leq c\gamma^{p/4},
\end{equation}
for $e_t$ as in \eqref{eq: e_t def}. By the chain rule again one has
\begin{align}\label{eq: chain rule calc}
 e^{p \mu t/4}E\lvert e_t\rvert^p &=\frac{p \mu}{4} \int^t_0 e^{\frac{p \mu}{4} s} \lvert e_s\rvert^p ds \nonumber \\
& -p\int^t_0 e^{\frac{p \mu}{4} s}\langle \nabla U(Y_s)-\nabla U(X_{\kappa_\gamma(s)}), Y_s-X_{\kappa_\gamma(s)}\rangle \lvert e_s\rvert^{p-2}ds \nonumber \\
&+ p\int^t_0 e^{\frac{p \mu}{4} s} \langle \nabla U(Y_s)-\nabla U(X_{\kappa_\gamma(s)}), X_s-X_{\kappa_\gamma(s)}\rangle \lvert e_s\rvert^{p-2} \nonumber \\
&=:\frac{p \mu}{4}\int^t_0 e^{\frac{p \mu}{4} s} \lvert e_s\rvert^p ds+w_1(t)+w_2(t).
\end{align}
One then calculates via the convexity assumption \Cref{assmp: strong mon}, the triangle inequality and Young's inequality that
\begin{align}
w_1(t)&\leq  -p\mu \int^t_0 \lvert Y_s-X_{\kappa_\gamma(s)}\rvert^2 \lvert e_s\rvert^{p-2}ds,
\end{align}
so writing $Y_s-X_{\kappa_\gamma(s)}=e_s+X_s-X_{\kappa_\gamma(s)}$ and expanding, one has
\begin{align}
w_1(t)\leq   \int^t_0 -p\mu e^{\frac{p \mu}{4} s}\lvert e_s\rvert^p-p\mu \lvert X_s-X_{\kappa_\gamma(s)}\rvert^2\lvert e_s\rvert^{p-2} -2p\mu \langle e_s, X_s-X_{\kappa_\gamma(s)} \rangle \lvert e_s\rvert^{p-2}ds ,
\end{align}
so that using Young's inequality (see Note \ref{note: youngs ineq}) one may bound the second term as
\begin{equation}
    -p\mu \lvert X_s-X_{\kappa_\gamma(s)}\rvert^2\lvert e_s\rvert^{p-2} \leq 0,
\end{equation}
and the third term as
\begin{equation}
-2p\mu \langle e_s, X_s-X_{\kappa_\gamma(s)} \lvert e_s\rvert^{p-2} \rangle \leq \frac{p\mu}{2}\lvert e_s\rvert^p+c\lvert X_s-X_{\kappa_\gamma(s)} \rvert^p,
\end{equation}
and therefore one obtains
\begin{align}
w_1(t)\leq   -\frac{p \mu}{2} \int^t_0  e^{\frac{p \mu}{4} s}\lvert e_s\rvert^pds+c\int^t_0 e^{\frac{p \mu}{4} s}\lvert X_s-X_{\kappa_\gamma(s)} \rvert^pds  .
\end{align}
Then applying Lemma \ref{lemma: incr bound} and integrating one has
\begin{equation}\label{eq: w1}
E w_1(t)\leq  -\frac{p \mu}{2}\int^t_0  e^{\frac{p \mu}{4} s}\lvert e_s\rvert^pds+c\gamma^{p/2}  e^{\frac{p \mu}{4} t}.
\end{equation}
Furthermore one uses Holders inequality to bound $w_2$ as
\begin{align}
E w_2(t)&\leq  \frac{p \mu}{4} \int^t_0 e^{\frac{p \mu}{4} s} E\lvert e_s\rvert^p ds +c\int^t_0 e^{\frac{p \mu}{4} s}E\lvert \langle \nabla U(Y_s)-\nabla U(X_{\kappa_\gamma(s)}), X_s-X_{\kappa_\gamma(s)}\rangle \rvert^{p/2} ds \nonumber \\
& \leq \frac{p \mu}{4} \int^t_0 e^{\frac{p \mu}{4} s} E\lvert e_s\rvert^p ds+ c\int^t_0 e^{\frac{p \mu}{4} s}(E\lvert  \nabla U(Y_s)-\nabla U(X_{\kappa_\gamma(s)}) \rvert^p)^{1/2} (E\lvert X_s-X_{\kappa_\gamma(s)}\rvert^p)^{1/2}.
\end{align}
Then one may bound the first factor in the second term by a constant, via the triangle inequality, the growth bound \Cref{assmp: growth assumption}, Lemmas \ref{lemma: moment bounds for cont} and \ref{lemma: moment bounds for scheme}, and the second factor by Lemma \ref{lemma: incr bound}, so that 
\begin{align}\label{eq: w2}
w_2(t)&\leq \frac{p \mu}{4} E \lvert e_t\rvert^p+c\gamma^{p/4}e^{\frac{p \mu}{4} t}.
\end{align}
It follows that one obtains \eqref{eq: uniform numerics bound 2} by substituting \eqref{eq: w1} and \eqref{eq: w2} into \eqref{eq: chain rule calc} and dividing through by $e^{\frac{p \mu}{4} t}$.
\end{proof}
\section{Examples}
We consider the case of Bayesian inference with a Gaussian prior and a Laplacian likelihood. Note that any convex optimisation problem with gradient satisfying \Cref{assmp: piecewise Lip} or \Cref{assmp: growth assumption} could be made to fit our assumptions by the addition of an $L^2$ regulariser to $U$. Our choice of Laplacian (or $\ell^1$) priors follows the example in \cite{doi:10.1137/16M1108340}, where they are used in the context of image reconstruction.
\subsection{Bayesian inference in one dimension}
Let us fix hyperparameters $b>0$, $\mu _0\in \mathbb{R}$ and $\sigma>0$, and suppose one has a prior distribution $\theta \sim \mathcal{N}(\mu_0, \sigma^2)$ and a likelihood $y \propto exp(-\frac{\lvert y -\theta \rvert }{b}) $. Then if one has observations $y_1, y_2,...,y_k \in \mathbb{R}$ of $y$, the Bayesian posterior for $\theta$ is
\begin{equation}\label{eq: posterior 1d}
p(\theta \vert y ) \sim \pi_\beta  \propto exp \biggr(-b^{-1}\sum_{i=1}^k \lvert y_i -\theta \rvert -\frac{( \theta - \mu_0)^2}{2\sigma^2} \biggr) ,
\end{equation}
for $U(\theta):= 2b^{-1}\sum_{i=1}^k \lvert y_i -\theta \rvert +\frac{1}{2\sigma^2}( \theta - \mu_0 )^2$, $\beta=1$ so that $\pi_1=\pi \propto e^{-U}$ as in the intoduction. Let us show that $U$ satisfies \Cref{assmp: piecewise Lip} and \Cref{assmp: strong mon}. Firstly note that $\nabla U$ exists everywhere except at $y_1, y_2,...,y_k \in \mathbb{R}$, so that since points are $1$-dimensional hypersurfaces one can set $\Sigma_i:=y_i$. Furthermore, for $\theta \in \mathbb{R}\setminus \{y_1, y_2,...,y_k\}$ one has
\begin{equation}
\nabla U(\theta) =2 b^{-1}\sum_{i=1}^k sign(\theta-y_i)+\sigma^{-2}(\theta-\mu_0),
\end{equation}
so that $\nabla U$ is clearly Lipschitz on all intervals between the $y_i$ and $\{-\infty, \infty\}$ and therefore $\Cref{assmp: piecewise Lip}$ is satisfied. Now let $\theta_1, \theta_2 \in \mathbb{R}\setminus \{y_1, y_2,...,y_k\}$ be arbitrary. Then since $\theta \mapsto sign(\theta-y_i)$ is increasing one has
\begin{align}
(\nabla U(\theta_1) - \nabla U(\theta_2)) (\theta_1-\theta_2)&=2b^{-1}\sum_{i=1}^k (sign(\theta_1-y_i)-sign(\theta_2-y_i))(\theta_1-\theta_2)+\sigma^{-2}(\theta_1-\theta)^2 \nonumber \\
& \geq \sigma^{-2}(\theta_1-\theta)^2 .
\end{align}
Therefore $\nabla U$ satisfies \Cref{assmp: strong mon} with $\mu=\sigma^{-2}$, and therefore providing one starts with initial condition $\xi$ such that $P(\xi=y_i)=0$ for $i=1,2,...,k$, one can apply Theorem \ref{th: main theorem} for \eqref{eq: ULA} applied for $U$ and $\beta$ given as above, in order to sample from \eqref{eq: posterior 1d}.

\subsection{Bayesian inference in higher dimensions}

Let us consider the same situation as above, but in $d$ dimensions, i.e. where one has a prior distribution $\theta \sim \mathcal{N}(\mu_0, \sigma^2 A)$, where $\mu_0\in \mathbb{R}^d$, $\sigma>0$ and $A\in \mathbb{R}^{d\times d}$ is a positive definite matrix with largest eigenvalue $1$, and a likelihood $y \propto exp(-\frac{\lvert y -\theta \rvert }{b}) $ for $y$ and $\theta$ taking values in $\mathbb{R}^d$. Then similarly to before
\begin{equation}\label{eq: posterior higher dim}
p(\theta \vert y ) \sim \pi_\beta  \propto exp \biggr(-b^{-1}\sum_{i=1}^k \lvert y_i -\theta \rvert -\frac{( \theta - \mu_0)^TA^{-1}( \theta - \mu_0)}{2\sigma^2} \biggr) ,
\end{equation}
so for $U(\theta):= 2b^{-1}\sum_{i=1}^k \lvert y_i -\theta \rvert +\frac{1}{2\sigma^2}( \theta - \mu_0)^TA^{-1}( \theta - \mu_0)$ one has
\begin{equation}
\nabla U(\theta) =2 b^{-1}\sum_{i=1}^k \frac{\theta-y_i}{\lvert \theta-y_i \rvert}+A^{-1}\sigma^{-2}(\theta-\mu_0).
\end{equation}
Therefore, since $A^{-1}$ must have smallest eigenvalue equal to $1$, one has
\begin{align}
\langle \nabla U(\theta_1) - \nabla U(\theta_2), \theta_1-\theta_2 \rangle &\geq 2b^{-1}\sum_{i=1}^k \langle \frac{\theta_1-y_i}{\lvert \theta_1-y_i \rvert}-\frac{\theta_2-y_i}{\lvert \theta_2-y_i \rvert}, \theta_1-\theta_2 \rangle+\sigma^{-2}\lvert \theta_1-\theta \rvert^2 .
\end{align}
and one calculates
\begin{align}
\langle \frac{\theta_1-y_i}{\lvert \theta_1-y_i \rvert}-\frac{\theta_2-y_i}{\lvert \theta_2-y_i \rvert}, \theta_1-\theta_2 \rangle &= \langle \frac{\theta_1-y_i}{\lvert \theta_1-y_i \rvert}-\frac{\theta_2-y_i}{\lvert \theta_2-y_i \rvert}, \theta_1-y_i-(\theta_2 -y_i) \rangle \nonumber \\
& = \lvert \theta_1-y_i\rvert+ \lvert \theta_2-y_i\rvert-\biggr (\frac{1}{\lvert \theta_1-y_i \rvert}+\frac{1}{\lvert \theta_2-y_i \rvert} \biggr)\langle \theta_1-y_i, \theta_2-y_i \rangle \nonumber \\
&\geq 0,
\end{align}
therefore proving $\nabla U$ obeys \Cref{assmp: strong mon} with $\mu=\sigma^{-2}$ as before. However this time one can show that $\nabla U$ does not obey \Cref{assmp: piecewise Lip} (since it is not piecewise-Lipschitz around any of the $y_i$) but instead the weaker assumption \Cref{assmp: growth assumption}, so in this situation one may apply Theorem \ref{th: main theorem 2} but not Theorem \ref{th: main theorem} to sample from \eqref{eq: posterior higher dim}.

\bibliographystyle{plain}
\bibliography{ref}
\end{document}